\documentclass[a4paper,11pt]{amsart}

\usepackage{amsfonts,latexsym,rawfonts,amsmath,amssymb,amsthm, mathrsfs, lscape}
\usepackage[all]{xy}
\usepackage[english]{babel}
\usepackage[utf8]{inputenc}
\usepackage{array, tabularx}
\usepackage{setspace}

\usepackage{textcomp}
\usepackage{multirow}
\usepackage{xfrac}

\usepackage[hypertexnames=false,
backref=page,
    pdftex,
    pdfpagemode=UseNone,
    breaklinks=true,
    extension=pdf,
    colorlinks=true,
    linkcolor=blue,
    citecolor=blue,
    urlcolor=blue,
]{hyperref}

\renewcommand{\Re}{\mathsf{Re}\,}
\renewcommand{\Im}{\mathsf{Im}\,}

\newcommand{\etalchar}[1]{$^{#1}$}

\newcommand{\referenza}{}

\newtheorem{thm}{Theorem}
\newtheorem*{thm*}{Theorem \referenza}
\newtheorem*{thms*}{Theorems \referenza}
\newtheorem{cor}[thm]{Corollary}
\newtheorem*{cor*}{Corollary \referenza}
\newtheorem{lem}[thm]{Lemma}
\newtheorem*{lem*}{Lemma \referenza}

\newtheorem{prop}[thm]{Proposition}
\newtheorem*{prop*}{Proposition \referenza}

\newtheorem{conj}[thm]{Conjecture}
\newtheorem*{conj*}{Conjecture \referenza}
\newtheorem{rmk}[thm]{Remark}

\newtheorem{defi}[thm]{Definition}

\numberwithin{equation}{section}

\def \R {\mathbb R}
\def \C {\mathbb C}

\def\frh{{\frak h}}

\def\XXint#1#2#3{{\setbox0=\hbox{$#1{#2#3}{\int}$ }
\vcenter{\hbox{$#2#3$ }}\kern-.6\wd0}}

\allowdisplaybreaks[3]

\title{On Gauduchon connections with K\"ahler-like curvature}

\author[D. Angella]{Daniele Angella}
\address[D. Angella]{Dipartimento di Matematica e Informatica ``Ulisse Dini''\\
Universit\`a di Firenze\\
via Morgagni 67/A\\
50134 Firenze, Italy}
\email{daniele.angella@unifi.it}
\email{daniele.angella@gmail.com}

\author[A. Otal]{Antonio Otal}
\address[A. Otal]{Centro Universitario de la Defensa\,-\,I.U.M.A.\\
Academia General Militar\\
Ctra. de Huesca s/n. 50090 Zaragoza\\
Spain}
\email{aotal@unizar.es}

\author[L. Ugarte]{Luis Ugarte}
\address[L. Ugarte]{Departamento de Matem\'aticas\,-\,I.U.M.A.\\
Universidad de Zaragoza\\
Campus Plaza San Francisco\\
50009 Zaragoza, Spain}
\email{ugarte@unizar.es}

\author[R. Villacampa]{Raquel Villacampa}
\address[R. Villacampa]{Centro Universitario de la Defensa\,-\,I.U.M.A.\\
Academia General Militar\\
Ctra. de Huesca s/n. 50090 Zaragoza\\
Spain}
\email{raquelvg@unizar.es}

\thanks{The first-named author is supported by Project PRIN ``Varietà reali e complesse: geometria, topologia e analisi armonica'', by Project FIRB ``Geometria Differenziale e Teoria Geometrica delle Funzioni'', by project SIR 2014 AnHyC ``Analytic aspects in complex and hypercomplex geometry'' code RBSI14DYEB, and by GNSAGA of INdAM.
The second-named, third-named, fourth-named authors are supported by the projects MTM2017-85649-P (AEI/FEDER, UE), and E22-17R “\'Algebra y Geometría”.
}

\date{\today}

\begin{document}

\begin{abstract}
We study Hermitian metrics with a Gauduchon connection being ``K\"ahler-like", namely, satisfying the same symmetries for curvature as the Levi-Civita and Chern connections.
In particular, we investigate $6$-dimensional solvmanifolds with invariant complex structures with trivial canonical bundle and with invariant Hermitian metrics. The results for this case give evidence for two conjectures that are expected to hold in more generality: first, if the Strominger-Bismut connection is K\"ahler-like, then the metric is pluriclosed; second, if another Gauduchon connection, different from Chern or Strominger-Bismut, is K\"ahler-like, then the metric is K\"ahler.
As a further motivation, we show that the K\"ahler-like condition for the Levi-Civita connection
assures that the Ricci flow preserves the Hermitian condition along analytic solutions.
\end{abstract}

\maketitle

\setcounter{tocdepth}{2} \tableofcontents

\section*{Introduction}

We study Hermitian manifolds whose curvature tensor, with respect to a ``canonical" connection, satisfies further symmetries, as initiated by A. Gray \cite{gray}, and by B. Yang and F. Zheng \cite{yang-zheng}. Symmetries clearly include the case of flat curvature \cite{boothby, wang-yang-zheng, ganchev-kassabov, discala-vezzoni, yang-zheng-Flat, vezzoni-yang-zheng, he-liu-yang}.
This is ultimately related to Yau's Problem 87 \cite{yau-problem} concerning compact Hermitian manifolds with holonomy reduced to subgroups of
${\mathrm U}(n)$.

The setting is the following. Let $X$ be a complex manifold endowed with a Hermitian metric.
One may wonder which connection is more ``natural" to investigate the complex geometry of $X$.

On the one side, the {\em Levi-Civita connection} $\nabla^{LC}$ is the only torsion-free metric connection. In general, it does not preserve the complex structure, this condition forcing the metric to be K\"ahler. The absence of torsion yields that its curvature tensor $R^{LC}$ satisfies the symmetry called {\em first Bianchi identity}:
$$ \sum_{\sigma \in \mathfrak{G}} R^{LC}(\sigma x, \sigma y) \sigma z = 0 . $$

On the other side, the {\em Chern connection} $\nabla^{Ch}$ is another tool to investigate the differential complex geometry of $X$.
It is the only Hermitian connection (namely, that preserves both the metric and the complex structure,) being compatible with the Cauchy-Riemann operator on the holomorphic tangent bundle. This yields that its curvature tensor $R^{Ch}$ satisfies the {\em type condition} symmetry:
$$ R^{Ch} \in \wedge^{1,1}(X;\mathrm{End}(T^{1,0}X)) . $$

The condition $R^{Ch}=R^{LC}$ forces the metric to be K\"ahler, as proven in \cite[Theorem 1.1]{yang-zheng}.
In the same spirit, in \cite[Corollary 4.5]{liu-yang}, the authors prove that if the total scalar curvatures $s_{Ch} = s_{LC}$, then the metric is balanced.
Notwithstanding, there are plenty of non-K\"ahler Hermitian metrics such that either $R^{LC}$ or $R^{Ch}$ satisfies both the symmetries of the Levi-Civita curvature and the symmetries of the Chern curvature. Such metrics are called {\em G-K\"ahler-like}, respectively {\em K\"ahler-like} by B. Yang and F. Zheng \cite{yang-zheng}.

Notice that no locally homogeneous examples of compact complex Hermitian surfaces with odd first Betti number exist such that the Levi-Civita connection is K\"ahler-like, thanks to a result by Muškarov \cite{muskarov} that prevents $J$-invariant Ricci tensor.

We point out three remarks that motivate the (G-)K\"ahler-like conditions. First, they imply the cohomological property of the metric being {\em balanced} in the sense of Michelsohn \cite{michelsohn}, as proven in \cite[Theorem 1.3]{yang-zheng}. This property plays a role in the Strominger system, and this is a first motivation in \cite{yang-zheng}. Second, symmetries are clearly satisfied by zero curvature: classification results in case of certain specific connections are investigated in \cite{boothby, wang-yang-zheng, ganchev-kassabov, yang-zheng-Flat, vezzoni-yang-zheng}.
This is related to Yau's Problem~87 in \cite{yau-problem} on compact Hermitian manifolds with holonomy reduced to a subgroup
of ${\mathrm U}(n)$ (here the connection needs not to be the Levi-Civita connection).
Third, we prove that the K\"ahler-like condition for the Levi-Civita connection assures that the Ricci flow preserves the Hermitian condition along (analytic, compare \cite{kotschwar}) solutions:

\renewcommand{\referenza}{\ref{thm:ricci-flow-hermitian}}
\begin{thm*}
Let $g_0$ be a Hermitian metric on a compact complex manifold, and consider an analytic solution $(g(t))_{t\in(-\varepsilon,\varepsilon)}$ for $\varepsilon>0$ of the Ricci flow
\begin{equation}\label{eq:ricci-flow}\tag{RF}
\frac{d}{dt}g(t) = -\mathrm{Ric}(g(t)), \qquad g(0)=g_0 .
\end{equation}
If the Levi-Civita connection of $g_0$ is K\"ahler-like, then $g(t)$ is Hermitian for any $t$.
\end{thm*}

In fact, there are infinitely-many Hermitian connections, namely, preserving both the metric and the complex structure, but possibly having torsion. Prescribing components of the torsion selects a special family of Hermitian connections $\nabla^{\varepsilon}$ called {\em canonical connections in the sense of Gauduchon} \cite{gauduchon-bumi}, varying $\varepsilon\in\mathbb R$. Such family includes and interpolates the Chern connection $\nabla^{Ch}=\nabla^{0}$ and the {\em Strominger-Bismut connection} $\nabla^{+}=\nabla^{\sfrac{1}{2}}$, both of which have a role in complex geometry and heterotic string theory \cite{strominger}.
In this note, we extend the notion of {\em K\"ahler-like} (Definition \ref{def:KL}) to the family $\{\nabla^{\varepsilon}\}_{\varepsilon\in\mathbb R}$ of Gauduchon connections.

Whereas the balanced condition is related to the Bott-Chern cohomology and to the Chern connection, the Strominger-Bismut connection is related to the Aeppli cohomology and to the notion of {\em pluriclosed} metrics (also known as SKT), see {\itshape e.g.} \cite{bismut}.
Therefore, as a Strominger-Bismut counterpart of \cite[Theorem 1.3]{yang-zheng} (see Theorem \ref{thm:yang-zheng-balanced}), the following conjecture is suggested:

\begin{conj}\label{conj:bismut-skt}
Consider a compact complex manifold endowed with a Hermitian metric.
If the Strominger-Bismut connection is K\"ahler-like, then the metric is pluriclosed.
\end{conj}

Notice that the particular case of Strominger-Bismut-flat metrics is studied in \cite{wang-yang-zheng}, where it is proven that a compact Hermitian manifold whose Strominger-Bismut connection is flat admits, as finite unbranched cover, a local Samelson space, given by the product of a compact semisimple Lie group and a torus, \cite[Theorem 1]{wang-yang-zheng}.

Moreover, they prove in \cite[Theorem 2]{wang-yang-zheng} that a balanced Strominger-Bismut-flat metric is actually K\"ahler. It is known that Hermitian metrics that are both pluriclosed and balanced are in fact K\"ahler, see {\itshape e.g.} \cite[Remark~1]{alexandrov-ivanov} or \cite[Proposition 1.4]{fino-parton-salamon}.
In fact, it is conjectured that no compact non-K\"ahler complex manifold can admit at the same time a balanced metric and a pluriclosed metric, see \cite[Problem 3]{fino-vezzoni}. This conjecture has proven to be true, for example, for the twistor spaces of
compact anti-self-dual Riemannian manifolds \cite{verbitsky}, for compact complex manifolds in the Fujiki class $\mathcal C$ \cite{chiose}, for $\sharp_k \mathbb S^3 \times \mathbb S^3$ for $k\geq 2$ endowed with the complex structures constructed from the conifold transition \cite{fu-li-yau}, for nilmanifolds with invariant complex structures \cite{fino-parton-salamon, fino-vezzoni, fino-vezzoni-pams}.
So this last theorem of Wang, Yang, and Zheng is explained by the
fact that the flatness of the Strominger-Bismut connection implies that the metric is pluriclosed,
which gives a partial evidence for the above Conjecture~\ref{conj:bismut-skt}
(see Section~\ref{relations} for details).

The canonical connections by Gauduchon allow to tie the Chern and Strominger-Bismut connections. So, one may claim that a generic canonical connection shares properties with both of them. Again since no non-K\"ahler metric can be both balanced and pluriclosed, we conjecture the following, that extends \cite[Conjecture 1.3]{yang-zheng-Flat} on Gauduchon-flat connections.

\begin{conj}\label{conj:gauduchon-kahler-like}
Consider a compact complex manifold endowed with a Hermitian metric.
Consider a canonical connection in the Gauduchon family, different from the Strominger-Bismut connection and the Chern connection.
If it is K\"ahler-like, then the metric is K\"ahler.
\end{conj}

We test these conjectures on a class of compact complex manifolds of complex dimension three given by quotients of Lie groups\footnote{
After the appearence of our note, recent progress on the Conjectures \ref{conj:bismut-skt} and \ref{conj:gauduchon-kahler-like} has been made by Zhao and Zheng \cite{zhao-zheng} and by Fu and Zhou \cite{fu-zhou}, respectively.
}. More precisely, we restrict to {\em Calabi-Yau solvmanifolds}, that is, compact quotients of solvable Lie groups by discrete subgroups,
endowed with an invariant complex structure having a non-zero invariant closed (3,0)-form,
which implies that the canonical bundle is holomorphically-trivial. We choose this class because of their role {\itshape e.g.} in constructing explicit invariant solutions to the Strominger system with respect to a Gauduchon connection \cite{fei-yau, OUV-strominger}.
They will be equipped with an invariant Hermitian metric. When the group is nilpotent, we get {\em nilmanifolds} and they always satisfy the latter condition on triviality of the canonical bundle by \cite{salamon, barberis-dotti-verbitsky}. Invariant complex structures on $6$-dimensional nilmanifolds are ultimately classified by \cite{couv}, see also the references therein, and they never admit K\"ahler metrics, unless the
nilmanifold be a torus, by topological obstructions \cite{benson-gordon, hasegawa}. Invariant complex structures on $6$-dimensional solvmanifolds with trivial canonical bundle are classified by \cite{fino-otal-ugarte}.

The following theorem is a consequence of direct inspection performed with the help of the symbolic computation softwares, Sage \cite{sage} and Mathematica.

\renewcommand{\referenza}{\ref{thm:solv}}
\begin{thm*}\label{thm:solvmanifolds}
Consider the class of compact $6$-dimensional solvmanifolds endowed with an invariant complex structure with holomorphically-trivial canonical bundle and with an invariant Hermitian metric.
\begin{itemize}
\item If the Strominger-Bismut connection is K\"ahler-like, then the metric is pluriclosed.
\item If a canonical connection in the Gauduchon family, different from the Strominger-Bismut connection and the Chern connection, is K\"ahler-like, then the metric is K\"ahler.
\end{itemize}
Therefore, Conjectures \ref{conj:bismut-skt} and \ref{conj:gauduchon-kahler-like} are satisfied for this class of manifolds.
\end{thm*}

A more detailed statement can be found in Theorem \ref{thm:solv-precise}, and Proposition~\ref{prop:lc} finalizes the study showing that, on six-dimensional solvmanifolds with holomorphically-trivial canonical bundle endowed with an invariant Hermitian non-K\"ahler metric, the Levi-Civita connection is never K\"ahler-like.

In particular, we observe here some evidence following from the performed study.

Holomorphically-parallelizable complex structures are Chern-flat and then their Chern connection is clearly K\"ahler-like.
Conversely, by \cite{boothby}, compact Hermitian Chern-flat manifolds are given by quotients of complex Hermitian Lie groups.
We have examples of non-Strominger-Bismut-flat metric whose Strominger-Bismut connection is K\"ahler-like.

The examples here confirm \cite[Conjecture 1.3]{yang-zheng-Flat}, stating that compact Hermitian manifolds with a flat Gauduchon connection (different from Chern or Strominger-Bismut) must be K\"ahler. The analogous question for invariant Hermitian structures on Lie groups is asked and investigated in \cite[Question 1.1]{vezzoni-yang-zheng}, and our result gives further evidence for a positive answer, since our computations are also valid on Lie groups endowed with left-invariant structures (see Remark \ref{rmk:non-cpt}).

On a compact complex manifold, being balanced is a necessary condition for either the Levi-Civita or the Chern connections being K\"ahler-like, thanks to \cite[Theorem 1.3]{yang-zheng}, but the converse does not hold true, as many examples below also show. It is expected that compact complex manifolds satisfying the $\partial\overline\partial$-Lemma admit balanced metrics. We ask {\itshape whether the $\partial\overline\partial$-Lemma property has any relation with the K\"ahler-like property}.

According to \cite[Problem 3]{fino-vezzoni}, no compact complex non-K\"ahler manifold is expected to bear both a balanced metric and a pluriclosed metric. Then, according to \cite[Theorem 1.3]{yang-zheng} and Conjecture \ref{conj:bismut-skt}, we expect that {\em a compact complex manifold admitting both a Hermitian metric with K\"ahler-like Chern connection and a (possibly different) Hermitian metric with K\"ahler-like Strominger-Bismut connection admits a K\"ahler metric}.
By Theorem~\ref{thm:solv},
this is obviously confirmed in the class of Calabi-Yau solvmanifolds of complex dimension $3$.
More in general, we ask {\itshape whether it is possible to admit two different connections
(with respect to possibly different metrics) being both K\"ahler-like}.

The Iwasawa manifold and its small deformations show that the property of being K\"ahler-like with respect to the Chern connection is not open under deformations of the complex structures.
It is not even closed, {\itshape i.e.} the central limit of a holomorphic family of compact Hermitian manifolds with the K\"ahler-like property may not admit any Hermitian metric satisfying the K\"ahler-like property with respect to the Chern connection, see Corollary~\ref{cor:kl-non-closed}. We ask {\itshape whether the K\"ahler-like property with respect to the Strominger-Bismut connection (or, more generally, any other Gauduchon connection,) is open and/or closed by holomorphic deformations of the complex structure}.

\bigskip

{\small
\noindent {\itshape Acknowledgments.}
This work has been written during several stays of the first-named author at the Departamento de Matem\'aticas de la Universidad de Zaragoza, which he thanks for the warm hospitality. He thanks also Jose Fernando for his warm hospitality at the Facultad de Matem\'aticas of the Universidad Complutense de Madrid.
The authors would like to thank Francesco Pediconi, Fabio Podestà, Cristiano Spotti, and Luigi Vezzoni for several interesting discussions.
Many thanks also to the anonymous Referees for their useful comments and suggestions.
}

\section{K\"ahler-like connections}

\subsection{Curvatures}
Let $X$ be a manifold endowed with a complex structure $J$ and a Hermitian structure $h=g-\sqrt{-1}\omega$.

Let $\nabla$ be any linear {\em metric} connection on $X$,  that is, preserving the Riemannian metric $g$, namely, $\nabla g=0$. (Denote with the same symbol its $\C$-bi-linear extension to $TX\otimes_\R\C$.) Consider its curvature operator:
$$ R^\nabla(x,y) := \nabla^2(x,y) = [\nabla_x,\nabla_y] - \nabla_{[x,y]} , $$
and define also the $(4,0)$-tensor
$$ R^\nabla(x,y,z,w) := g(R^\nabla(x,y)z,w) . $$
Notice that, by the very definition, respectively by the condition $\nabla g=0$,
$$ R^\nabla(\_,\_) \in \wedge^2(X;\mathrm{End}(TX))
\qquad
\text{ and }
\qquad
R^\nabla(\_,\_,\_,\_) \in \wedge^2X\otimes\wedge^2X .$$

The {\em first Ricci curvature} and the {\em second Ricci curvature} are, respectively, the traces
$$ \mathrm{Ric}^{(i)}(x,y) = \mathrm{tr}\,  R^\nabla(x,y) \in \wedge^{2}X, $$
$$ \mathrm{Ric}^{(ii)}(z,w) = \mathrm{tr}_g  R^\nabla(\_,\_,z,w) \in \wedge^2X \subset \mathrm{End}(TX) . $$
The {\em scalar curvature} is
$$ \mathrm{Scal} = \mathrm{tr}_g \mathrm{Ric}^{(i)} = \mathrm{tr}\, \mathrm{Ric}^{(ii)} \in\mathcal{C}^\infty(X;\R) . $$

\subsection{Gauduchon connections}\label{subsec:gauduchon-connections}
We  focus on the {\em canonical connections} in the Gauduchon family as defined in \cite[Definition 2]{gauduchon-bumi}: for $t\in\mathbb R$, the  Hermitian connection $\nabla^{G_t}=:\nabla^{\frac{1-t}{4}, \frac{1+t}{4}}$ associated to $(J,g)$ is defined by
$$ g(\nabla^{G_t}_xy, z) = g(\nabla^{LC}_xy, z) + \frac{1-t}{4}\, T(x,y,z) + \frac{1+t}{4}\, C(x,y,z) , $$
where
$$ T := Jd \omega := -d \omega(J\_,J\_,J\_), \qquad C := d \omega(J\_,\_,\_) , $$
and $\nabla^{LC}$ denotes the Levi-Civita connection, that we can compute by
\begin{eqnarray*}
g(\nabla^{LC}_xy, z)
&=& \frac{1}{2}\,\left(
x\,g(y,z) + y\,g(x,z) - z\,g(x,y) \right. \\[5pt]
&& \left. + g([x,y],z) - g([y,z],x) - g([x,z],y) \right) .
\end{eqnarray*}
Note that, in the space of metric connections $\nabla^{\varepsilon,\rho}$ as defined in \cite{OUV-strominger}, Gauduchon connections sit on the line $\varepsilon+\rho=\frac{1}{2}$.

Other than
$$ \nabla^{LC} = \nabla^{0,0} , $$
special values are
$$
\nabla^{Ch} = \nabla^{G_1} = \nabla^{0,\frac{1}{2}} ,
\qquad
\nabla^+ = \nabla^{G_{-1}} = \nabla^{\frac{1}{2},0} ,
\qquad
\nabla^- = \nabla^{-\frac{1}{2},0}
$$
corresponding to the Chern, Strominger-Bismut, and anti-Strominger-Bismut connections, respectively.
Moreover,
$$\nabla^{G_0}=\nabla^{\frac{1}{4},\frac{1}{4}}=:\nabla^{lv}$$
is the {\em first canonical connection}, also called {\em associated connection} \cite{gauduchon-bumi, ganchev-kassabov}, namely the projection onto the holomorphic tangent bundle of the Levi-Civita connection.
Another important connection is given by
$$\nabla^{G_{\sfrac{1}{3}}},$$
called the {\em minimal Gauduchon connection} \cite{gauduchon-bumi, yang-zheng-Flat} because it is distinguished by the property that it has the smallest total torsion among Gauduchon connections.

\subsection{Symmetries of the curvature}

If we focus on the {\em Levi-Civita connection} $\nabla^{LC}$, (namely, the unique metric torsion-free connection on $X$,) respectively on the {\em Chern connection} $\nabla^{Ch}$, (namely, the unique Hermitian connection being compatible with the Cauchy-Riemann operator $\overline\partial$ on the holomorphic tangent bundle,) we have further symmetries.

The Levi-Civita connection $\nabla:=\nabla^{LC}$ satisfies the {\em first Bianchi identity}:
\begin{equation}\label{eq:bianchi1}\tag{1Bnc}
\sum_{\sigma\in\mathfrak{G}} R^{\nabla}(\sigma x,\sigma y)\sigma z = 0 .
\end{equation}
(This follows by the Levi-Civita connection being torsion-free, where the torsion is defined as
$$ T^\nabla(x,y) := \nabla_xy-\nabla_yx-[x,y] , $$
and by the Jacobi identity for the Lie bracket $[\_,\_]$.)
More in general, for a metric connection $\nabla$ with possibly non-zero torsion $T^\nabla$, we have (compare {\itshape e.g.} \cite[\S1.16]{gauduchon-book}):
$$ \sum_{\sigma\in\mathfrak{G}} R^{\nabla}(\sigma x,\sigma y)\sigma z = d^\nabla T(x,y,z) . $$

The Chern connection $\nabla:=\nabla^{Ch}$ satisfies the {\em type condition}:
$$ R^{\nabla} \in \wedge^{1,1}(X;\mathrm{End}(T^{1,0}X)) , $$
namely,
\begin{equation}\label{eq:type}\tag{Cplx}
R^{\nabla}(x,y,z,w) \;=\; R^{\nabla}(x,y,Jz,Jw) \;=\; R^{\nabla}(Jx,Jy,z,w)\;.
\end{equation}
(The $J$-invariance in the third and fourth arguments, namely the property $R^\nabla\in\wedge^2(X;\mathrm{End}(T^{1,0}X))$, follows from $\nabla J=0$ and $g(J\_,J\_)=g(\_,\_)$. The  conclusion follows from $\nabla^{0,1}=\overline\partial$ yielding $(\nabla^{0,1})^2=0$ and by $\nabla$ being real.)

\begin{rmk}
Condition \eqref{eq:bianchi1} for a metric connection implies
\begin{equation}\label{eq:symm}\tag{Symm}
R^\nabla \in S^2\wedge^2X ,
\end{equation}
namely, $R^{\nabla}(x,y,z,w) = R^{\nabla}(z,w,x,y)$.
In this case, the first Bianchi identity holds for the $(4,0)$-tensor $R^\nabla$ for permutations of any triple of indices.
If moreover $R^\nabla \in \wedge^{2}(X;\mathrm{End}(T^{1,0}X))$, then actually $R^\nabla \in \wedge^{1,1}(X;\mathrm{End(T^{1,0}X))}$.
\end{rmk}

\subsection{K\"ahler-like symmetries}
On a K\"ahler manifold, we have
$$ \nabla^{LC} \;=\; \nabla^{Ch} \;, $$
since $\nabla^{LC}J=0$. In particular, $\nabla := \nabla^{LC}$ satisfies both \eqref{eq:bianchi1} and \eqref{eq:type}.

\begin{defi}\label{def:KL}
Let $X$ be a complex manifold endowed with a Hermitian structure. Let $\nabla$ be a metric connection on it.
We say that $\nabla$ is {\em K\"ahler-like} if it satisfies \eqref{eq:bianchi1} and \eqref{eq:type}.
\end{defi}

\begin{rmk}[comparison with Yang and Zheng's \cite{yang-zheng}]
We claim that the above definition specializes to Yang and Zheng's definitions in \cite{yang-zheng} in case of Chern connection and Levi-Civita connection. Recall that, in \cite[page 2]{yang-zheng}, B. Yang and F. Zheng call a Hermitian structure:
\begin{itemize}
\item {\em K\"ahler-like (in the sense of \cite{yang-zheng})} if, for any $(1,0)$-tangent vector $X$, $Y$, $Z$, and $W$, it holds
$$ R^{Ch}(X, \bar Y, Z, \bar W) \;=\; R^{Ch}(Z, \bar Y, X, \bar W) \;; $$
\item {\em G-K\"ahler-like (in the sense of \cite{yang-zheng})}, in honour of Gray, if, for any $(1,0)$-tangent vector $X$, $Y$, $Z$, and $W$, it holds
$$ R^{LC}(X, Y, \bar Z, \bar W) \;=\; R^{LC}(X, Y, Z, \bar W) \;=\; 0 \;. $$
\end{itemize}
Then:
\begin{enumerate}
\item a Hermitian structure is K\"ahler-like in the sense of \cite{yang-zheng} if and only if its Chern connection is K\"ahler-like in the sense of Definition~\ref{def:KL};
\item a Hermitian structure is G-K\"ahler-like in the sense of \cite{yang-zheng} if and only if its Levi-Civita connection is K\"ahler-like in the sense of Definition~\ref{def:KL}.
\end{enumerate}
\end{rmk}

\begin{proof}
{\itshape (1)} Assume that the Chern connection $\nabla^{Ch}$ is K\"ahler-like in the sense of Definition \ref{def:KL}. In particular, by \eqref{eq:bianchi1}, we get:
\begin{eqnarray*}
R^{Ch}(X, \bar Y, Z, \bar W) &=& -R^{Ch}(\bar Y, Z, X, \bar W) - R^{Ch}(Z, X, \bar Y, \bar W) \\[5pt]
&=& R^{Ch}(Z, \bar Y, X, \bar W) + g(R^{Ch}(X, Z)\bar Y, \bar W) \\[5pt]
&=& R^{Ch}(Z, \bar Y, X, \bar W) ,
\end{eqnarray*}
since $g(R^{Ch}(X, Z)\bar Y, \bar W)=0$.

Conversely, assume that the Hermitian structure is K\"ahler-like in the sense of \cite{yang-zheng}. We already noticed that the Chern connection satisfies \eqref{eq:type}. We prove \eqref{eq:bianchi1}. It suffices to prove it for
$$ (x,y,z,w) \in \{ (\bar X, Y, Z, \bar W),\, (X, \bar Y, Z, \bar W),\, (X, Y, \bar Z, \bar W) \} , $$
where $X$, $Y$, $Z$, and $W$ are $(1,0)$-tangent vector fields. For example, in the first case, we have:
\begin{eqnarray*}
\lefteqn{ R^\nabla(\bar X, Y, Z, \bar W)+R^\nabla(Y, Z, \bar X, \bar W)+R^\nabla(Z, \bar X, Y, \bar W) } \\[5pt]
&=& -R^\nabla(Y, \bar X, Z, \bar W)+R^\nabla(Z, \bar X, Y, \bar W) = 0 ,
\end{eqnarray*}
where we used the skew-symmetries of $R^\nabla$, the type properties of $R^\nabla$, and the condition of K\"ahler-like in the sense of \cite{yang-zheng}.

\medskip

{\itshape (2)} Assume that the Levi-Civita connection $\nabla^{LC}$ is K\"ahler-like in the sense of Definition \ref{def:KL}. In particular, by \eqref{eq:type}, we get: for any $(1,0)$-tangent vector fields $X$, $Y$, $Z$, and $W$,
$$ R^{LC}(X, Y, \bar Z, \bar W) = R^{LC}(X, Y, J\bar Z, J\bar W) = -R^{LC}(X, Y, \bar Z, \bar W) , $$
whence
$$ R^{LC}(X, Y, \bar Z, \bar W) = 0 \;; $$
moreover, by using \eqref{eq:symm}, we get:
\begin{eqnarray*}
R^{LC}(X, Y, Z, \bar W) &=& R^{LC}(Z, \bar W, X, Y) = R^{LC}(Z, \bar W, JX, JY) \\[5pt]
&=& -R^{LC}(Z, \bar W, X, Y) = - R^{LC}(X, Y, Z, \bar W) ,
\end{eqnarray*}
whence
$$ R^{LC}(X, Y, Z, \bar W) = 0 . $$

Conversely, assume that the Hermitian structure is G-K\"ahler-like in the sense of \cite{yang-zheng}. By \cite{gray}, we have also: for any $(1,0)$-tangent vector fields $X$, $Y$, $Z$, and $W$,
\begin{equation}\label{eq:gray}\tag{Gray}
R^{LC}(X,Y,Z,W) = 0 .
\end{equation}
We already noticed that the Levi-Civita connection satisfies \eqref{eq:bianchi1}. We prove \eqref{eq:type}. By \eqref{eq:symm}, it suffices to show $R(x,y,z,w)=R(z,y,Jz,Jw)$. By the definition of G-K\"ahler-like and by \eqref{eq:gray}, and up to conjugation and (skew-)symmetries, we are reduced to the terms of curvatures:
\begin{eqnarray*}
R(\bar X, Y, Z, \bar W) = -R(Y, \bar X, Z, \bar W) , & \qquad & R(X, \bar Y, Z, \bar W) , \\[5pt]
R(X, Y, Z, \bar W) = 0 , & \qquad & R(X, Y, \bar Z, \bar W) = 0 , \\[5pt]
R(\bar X, Y, \bar Z, \bar W) = 0 , & \qquad & R(X, \bar Y, \bar Z, \bar W) = 0 , \\[5pt]
R(\bar X, \bar Y, Z, \bar W) = 0 , & \qquad & R(\bar X, \bar Y, \bar Z, \bar W) = 0 ,
\end{eqnarray*}
whence to the case: $(x,y,z,w)=(X, \bar Y, Z, \bar W)$. Of course, in this case \eqref{eq:type} holds.
\end{proof}

We will use also the following characterization of the K\"ahler-like condition.

\begin{rmk}\label{caracterizacionKL}
A metric linear connection $\nabla$ on a manifold endowed with a Hermitian structure is K\"ahler-like in the sense of Definition \ref{def:KL} if and only if, for any $(1,0)$-tangent vector fields $X$, $Y$, $Z$, and $W$, for any tangent vector fields $x$, $y$, $z$, and $w$,
\begin{equation}\label{eq:bianchi1-v2}\tag{1Bnc'}
B(X, \bar Y, Z, \bar W) := R^\nabla(X, \bar Y, Z, \bar W) - R^\nabla(Z, \bar Y, X, \bar W) = 0 ,
\end{equation}
\begin{equation}\label{eq:type-v2}\tag{Cplx'}
R^\nabla(x, y, Z, W) = R^\nabla(X, Y, z, w) = 0 .
\end{equation}
\end{rmk}

\begin{proof}
Assume that $\nabla$ is K\"ahler-like in the sense of Definition \ref{def:KL}. Moreover, we get also \eqref{eq:symm}. Then:
\begin{eqnarray*}
R^\nabla(x, y, Z, W) &=& R^\nabla(x, y, JZ, JW) \;=\; -R^\nabla(x, y, Z, W) \;, \\[5pt]
R^\nabla(X, Y, z, w) &=& R^\nabla(z, w, X, Y) \;=\; R^\nabla(z, w, JX, JY) \\[5pt]
&=& -R^\nabla(z, w, X, Y) \;=\; -R^\nabla(X, Y, z, w) \;,
\end{eqnarray*}
whence we get \eqref{eq:type-v2}.
Now, we have:
\begin{eqnarray*}
R^\nabla(X, \bar Y, Z, \bar W) &=& -R^\nabla(\bar Y, Z, X, \bar W)-R^\nabla(Z, X, \bar Y, \bar W) \\[5pt]
&=& R^\nabla(Z, \bar Y, X, \bar W) \;,
\end{eqnarray*}
that is, \eqref{eq:bianchi1-v2}.

On the other hand, up to anti-symmetries and conjugation, the only non-zero terms of the curvature are
$$ R^\nabla(X, \bar Y, Z, \bar W) \;. $$
Therefore, \eqref{eq:type} is satisfied. Also, it follows that \eqref{eq:bianchi1-v2} yields \eqref{eq:bianchi1}.
\end{proof}

\begin{rmk}
We notice  that the possible several traces become equal as soon as we assume further symmetries of the curvature tensor.
In this direction, compare the results by Liu and Yang in \cite{liu-yang}, comparing scalar curvatures. In particular, \cite[Corollaries 4.4 and 4.5]{liu-yang} state conditions concerning equalities between total scalar curvatures that force the metric to be K\"ahler, respectively balanced.
\end{rmk}

\subsection{Relations to special Hermitian metrics}\label{relations}
A Hermitian metric $\omega$ on a complex manifold is called {\em balanced} \cite{michelsohn} if $d^*_\omega\omega=0$, where $d^*_\omega$ denotes the adjoint operator of $d$ with respect to the $L^2$-pairing induced by the metric associated to $\omega$; in other words, $d\omega^{n-1}=0$, where $n$ is the complex dimension of the manifold.

We recall the following obstruction to being K\"ahler-like, respectively G-K\"ahler-like in the sense of \cite{yang-zheng}.
\begin{thm}[{\cite[Theorem 1.3]{yang-zheng}}]\label{thm:yang-zheng-balanced}
If a Hermitian metric on a compact complex manifold has either Levi-Civita or Chern connection being K\"ahler-like, then it is balanced in the sense of Michelsohn \cite{michelsohn}.
\end{thm}

A Hermitian metric $\omega$ on a complex manifold is called {\em pluriclosed} (also known as {\em SKT})
if $\partial\overline\partial\omega=0$ \cite{bismut}. Using a result in \cite{wang-yang-zheng}, and
inspired by the argument in \cite{yang-zheng-Flat}, we show the following result for Strominger-Bismut-flat Hermitian manifolds,
which can be seen as an evidence for Conjecture~\ref{conj:bismut-skt} in the special case of flat connection.

\begin{thm}\label{thm:bismut-flat-skt}
Let $X$ be a complex manifold of complex dimension $n$ endowed with a Hermitian metric.
If the Strominger-Bismut connection is flat, then the metric is pluriclosed.
\end{thm}

\begin{proof}
Denote by $\omega$ the Strominger-Bismut-flat Hermitian metric on the manifold $X$. 
The pluriclosed condition $\partial\overline\partial\omega=0$ is a local property.
The Strominger-Bismut connection $\nabla^+$ being flat, we can choose an unitary parallel frame $(e_j, \bar e_j)_j$ in a neighbourhood of any point
$p$ of $X$,
that is, $\nabla^+e_j=\nabla^+\bar e_j=0$. Furthermore, by \cite[Proof of Theorem 1, page 14]{wang-yang-zheng}, it has (local) structure constants
$$
[e_i,e_j]=c_{ij}^k e_k,
\qquad
[e_i, \bar e_j] = \overline{c_{jk}^i} e_k - c_{ik}^j \bar e_k ,
$$
where $c_{ij}^k\in\mathbb C$ are constants satisfying the required identities,
and $i, j, k\in\{1,\ldots, n\}$.
In addition, the metric satisfies 
$$\langle[X,Y],Z\rangle=-\langle[X,Z],Y\rangle,$$ 
for any
$X,Y,Z$ in
$(e_j, \bar e_j)_{1\leq j\leq n}$. 
Hereafter, we shorten {\itshape e.g.} $e^{i\bar j}:=e^i\wedge \bar e^j$, where $(e^j,\bar e^j)_j$
is the dual coframe of $(e_j,\bar e_j)_j$, and we use Einstein notation on summation over repeated indices.
In this notation, we have that with respect to the unitary parallel frame $(e_j, \bar e_j)_j$ 
the metric $\omega$ is expressed locally as
$$\omega = \frac{\sqrt{-1}}{2} e^{k \bar k}.$$

In other words, we are locally reduced to the Lie algebra case 
where the complex structure and the metric are both left invariant, and moreover, 
the metric is actually bi-invariant. 
In these conditions one has $\partial\overline\partial\omega=0$ 
\cite{SSTV88} (see also \cite[page 568]{enrietti}), so the metric $\omega$ is pluriclosed.
\end{proof}

\section{Hermitian geometry on six-dimensional Calabi-Yau solvmanifolds}
We consider $6$-dimensional {\em solvmanifolds}  $X = \left. \Gamma \middle\backslash G \right.$, namely, compact quotients of connected simply-connected solvable Lie groups by co-compact discrete subgroups. In particular, we consider solvmanifolds endowed with invariant Hermitian structures $(J,g)$, that is, structures whose lift to the universal cover is invariant under left-translations. Such structures are encoded in the associated Lie algebra $\mathfrak{g}$.
As a matter of notation, we will denote Lie algebras in the concise form in \cite{salamon}: {\itshape e.g.} $\mathfrak{rh}_{3}=(0,0,-12,0)$ means that the four-dimensional Lie algebra $\mathfrak{rh}_{3}$ admits a basis $(e_1,e_2,e_3,e_4)$ such that $[e_1,e_2]=e_3$, the other brackets being trivial. The notation actually refers to the dual $\mathfrak{rh}_{3}^\ast$: the dual basis $(e^1, e^2,e^3,e^4)$ satisfies $de^1=de^2=de^4=0$ and $de^3=-e^1\wedge e^2$, where we will shorten $e^{12} := e^1\wedge e^2$. Here, we use the following formula to relate the differential $d\colon \mathfrak{g}^\vee\to\wedge^2\mathfrak{g}^\vee$ and the bracket $[\_,\_]\colon\wedge^2\mathfrak{g}\to\mathfrak{g}$:
$$ d\alpha(x,y)=-\alpha([x,y]) . $$

\subsection{Metric formulas}
Let
$$ \left( \varphi^1, \varphi^2, \varphi^3 \right) $$
be an invariant co-frame of $(1,0)$-forms with respect to $J$.
With respect to the frame $\left( \varphi_1, \varphi_2, \varphi_3, \bar\varphi_1, \bar\varphi_2, \bar\varphi_3 \right)$ dual
to $\left( \varphi^1, \varphi^2, \varphi^3, \bar\varphi^1, \bar\varphi^2, \bar\varphi^3 \right)$, we first set the structure constants
$$ [ \varphi_I, \varphi_H ] =: c_{IH}^K \varphi_K .$$
(Capital letters here vary in $\{1,2,3,\bar1,\bar2,\bar3\}$ and refer to the corresponding component; the Einstein summation is assumed.)

The generic invariant Hermitian structure $\omega=g(\_,J\_)$ is given by
\begin{eqnarray}\label{eq:metric}
2\,\omega &=& \sqrt{-1} r^2\, \varphi^{1\bar1} + \sqrt{-1} s^2\,\varphi^{2\bar2} + \sqrt{-1} t^2\, \varphi^{3\bar3} \\[5pt]
&& + u\, \varphi^{1\bar2} - \bar u\, \varphi^{2\bar1} + v\, \varphi^{2\bar3} - \bar v\, \varphi^{3\bar2} + z\, \varphi^{1\bar3} - \bar z\, \varphi^{3\bar1} \nonumber
\end{eqnarray}
where the coefficients satisfy \cite[page 189]{ugarte-transf}
$$ r^2>0 ,\qquad s^2>0 , \qquad t^2 > 0 , $$
$$
r^2s^2 > |u|^2 , \qquad
r^2t^2 > |z|^2 , \qquad
s^2t^2 > |v|^2 ,
$$
$$
r^2s^2t^2 + 2 \Re(\sqrt{-1}\,\bar u z \bar v) > t^2|u|^2+r^2|v|^2+s^2|z|^2 .
$$
(Hereafter, we shorten {\itshape e.g.} $\varphi^{3\bar1}:=\varphi^3\wedge\bar\varphi^1$.)
That is to say, the matrix $(g_{KL})_{K,L}$ is positive-definite. Its inverse is denoted by
$$ (g^{KL})_{K,L} := (g_{KL})_{K,L}^{-1} . $$

Let $\Omega$ be the matrix associated to the Hermitian metric, i.e.
$$
\Omega=\left(\!\!\! \begin{array}{ccc}
\sqrt{-1}\, \frac{r^2}{2} & \frac{u}{2} & \frac{z}{2} \\
-\frac{\overline{u}}{2} & \sqrt{-1}\, \frac{s^2}{2} & \frac{v}{2} \\
-\frac{\overline{z}}{2} & -\frac{\overline{v}}{2} & \sqrt{-1}\, \frac{t^2}{2}
\end{array} \!\right)\!
$$
Then,
\begin{equation}\label{det}
8\sqrt{-1}\, \det\Omega = r^2s^2t^2-r^2|v|^2-s^2|z|^2-t^2|u|^2 + 2\,\Re(\sqrt{-1}\,\bar u\bar v z) > 0.
\end{equation}

We can express the Christoffel symbols of $\nabla^{LC}$ as
\begin{eqnarray*}
(\Gamma^{LC})_{IH}^{K} &=& \frac{1}{2}\, g^{KL}\, \left( g([\varphi_I,\varphi_H], \varphi_L) - g([\varphi_H,\varphi_L], \varphi_I) - g([\varphi_I,\varphi_L], \varphi_H) \right) \\[5pt]
&=& \frac{1}{2}\, c_{IH}^{K} - \frac{1}{2}\, g^{KA}g_{BI}c_{HA}^{B} - \frac{1}{2}\, g^{KA}g_{BH}c_{IA}^{B} .
\end{eqnarray*}
Therefore, for $(\varepsilon,\rho)$, we compute the Christoffel symbols of the $\nabla^{\varepsilon,\rho}$ connection:
$$ (\Gamma^{\varepsilon,\rho})_{IH}^{K} = (\Gamma^{LC})_{IH}^{K} + \varepsilon g^{KL} T_{IHL} + \rho g^{KL} C_{IHL} , $$
where
$$ T_{IHL} = -d\omega (J\varphi_I, J\varphi_H, J\varphi_L), \qquad C_{IHL} = d\omega(J\varphi_I, \varphi_H, \varphi_L) . $$
We can then express the $(4,0)$-Riemannian curvature of $\nabla^{\varepsilon,\rho}$ as
\begin{eqnarray*}
(R^{\varepsilon,\rho})_{IHKL} &=& g_{AL}(\Gamma^{\varepsilon,\rho})_{HK}^{B}(\Gamma^{\varepsilon,\rho})_{IB}^{A} - g_{AL} (\Gamma^{\varepsilon,\rho})_{IK}^{B} (\Gamma^{\varepsilon,\rho})_{HB}^{A} \\[5pt]
&& - g_{AL} c_{IH}^{B} (\Gamma^{\varepsilon,\rho})_{BK}^{A} .
\end{eqnarray*}

\subsection{Complex structures}\label{subsec:complex-st}
We consider $6$-dimensional {\em Calabi-Yau solvmanifolds}, meaning that they are endowed with an invariant complex structure
having a non-zero invariant closed (3,0)-form, thus
their canonical bundle is holomorphically-trivial. This includes nilmanifolds \cite{salamon, ugarte-transf, couv} and the solvmanifolds in \cite{fino-otal-ugarte}. We recall here their classification.

Six-dimensional nilpotent Lie algebras have been classified by Morozov in $34$ different classes, $18$ of which admit invariant complex structures by Salamon \cite{salamon}: they are $\mathfrak{h}_1$, \dots, $\mathfrak{h}_{16}$, $\mathfrak{h}_{19}^{-}$, $\mathfrak{h}_{26}^{+}$. The complex structures on them are classified into four families in \cite{couv}. We recall the complex structure equations and the underlying Lie algebras in Table \ref{table:nil-cplx}.

\begin{table}[h!]
\renewcommand*{\arraystretch}{1.6}
\begin{center}
{\resizebox{\textwidth}{!}{
\begin{tabular}{|c|c|l|}
\hline
\textbf{Name}&\textbf{Complex structure}&\textbf{Lie algebra}\\ \hline\hline
\multirow{2}{*}{\text{(Np)}}&\multirow{2}{*}{$d\varphi^1=d\varphi^2 = 0,\,\, d\varphi^3=\rho\, \varphi^{12},\,\, \text{ where }\rho\in\{0,1\}$}& $\rho=0: \frh_1=(0,0,0,0,0,0)$  \\ \cline{3-3}
&&$\rho=1: \frh_5=(0,0,0,0,13+42,14+23)$\\ \hline  \hline
\multirow{6}{*}{\text{(Ni)}}&\multirow{2}{*}{$d\varphi^1=d\varphi^2 = 0$,}&$\frh_2=(0,0,0,0,12,34)$ \\ \cline{3-3}
&& $\frh_3=(0,0,0,0,0,12+34)$ \\ \cline{3-3}
&$d\varphi^3= \rho\, \varphi^{12} + \varphi^{1\bar1} + \lambda\,\varphi^{1\bar 2} + D\,\varphi^{2\bar2}$,& $\frh_4=(0,0,0,0,12,14+23)$ \\ \cline{3-3}
&& $\frh_5=(0,0,0,0,13+42,14+23)$ \\ \cline{3-3}
&where $\rho\in\{0,1\}, \lambda\in\mathbb R^{\geq0}, D\in\mathbb C \text{ with }\Im D\geq0 $& $\frh_6=(0,0,0,0,12,13)$ \\ \cline{3-3}
&& $\frh_8=(0,0,0,0,0,12)$ \\ \hline \hline
\multirow{9}{*}{\text{(Nii)}}&\multirow{3}{*}{$d\varphi^1=0, \quad d\varphi^2=\varphi^{1\bar1}$,}&$\frh_7=(0,0,0,12,13,23)$ \\ \cline{3-3}
&& $\frh_9=(0,0,0,0,12,14+25)$ \\ \cline{3-3}
&& $\frh_{10}=(0,0,0,12,13,14)$ \\ \cline{3-3}
&\multirow{2}{*}{$d\varphi^3=\rho\varphi^{12}+B\,\varphi^{1\bar2}+c\,\varphi^{2\bar1}$,}& $\frh_{11}=(0,0,0,12,13,14+23)$ \\ \cline{3-3}
&& $\frh_{12}=(0,0,0,12,13,24)$ \\ \cline{3-3}
&\multirow{3}{*}{where $\rho\in\{0,1\}, B \in \mathbb C, c\in\mathbb R^{\geq0} , \text{ with } (\rho,B,c)\neq(0,0,0)$}& $\frh_{13}=(0,0,0,12,13+14,24)$ \\ \cline{3-3}
&& $\frh_{14}=(0,0,0,12,14,13+42)$ \\ \cline{3-3}
&& $\frh_{15}=(0,0,0,12,13+42,14+23)$\\ \cline{3-3}
&& $\frh_{16}=(0,0,0,12,14,24)$ \\ \hline \hline
\multirow{2}{*}{\text{(Niii)}}&$d\varphi^1=0,\quad d\varphi^2 =\varphi^{13} + \varphi^{1\bar 3}$,&$\frh_{19}^-=(0,0,0,12,23,14-35)$ \\ \cline{3-3}
&$d\varphi^3=\sqrt{-1}\rho\, \varphi^{1\bar1}\pm \sqrt{-1}(\varphi^{1\bar2} - \varphi^{2\bar1}),\quad \text{ where }\rho\in\{0,1\}$&$\frh_{26}^+=(0,0,12, 13,23,14+25)$ \\ \hline  \hline
 \end{tabular}
 }}
 \caption{Invariant complex structures on six-dimensional nilmanifolds up to linear equivalence,
 see \cite{andrada-barberis-dotti, ugarte-villacampa-Asian, couv}.}
 \label{table:nil-cplx}
\end{center}
\end{table}

We also consider  solvmanifolds other than nilmanifolds. Classification of invariant complex structures in dimension $6$ such that the canonical bundle is holomorphically-trivial is obtained in \cite{otal-thesis, fino-otal-ugarte}.
We recall the complex structure equations and the underlying Lie algebras in Table \ref{table:solv-cplx}.

\begin{table}[h!]
\renewcommand*{\arraystretch}{1.6}
\begin{center}
{\resizebox{\textwidth}{!}{
\begin{tabular}{|c|c|l|}
\hline
\textbf{Name}&\textbf{Complex structure}&\textbf{Lie algebra}\\ \hline\hline
\multirow{3}{*}{\text{(Si)}}&$d\varphi^1=A\varphi^{13}+A\varphi^{1\bar3}$,&$\mathfrak{g}_1=(15,-25,-35,45,0,0)$  when $\theta=0$  \\ \cline{3-3} 
&$d\varphi^2=-A\varphi^{23}-A\varphi^{2\bar3},\quad d\varphi^3=0$,& $\mathfrak{g}_2^{\alpha}=(\alpha \times 15+25, -15+\alpha\times 25, -\alpha\times 35+45, -35-\alpha\times 45,0,0)$\\
&where $A = \cos\theta+\sqrt{-1}\sin\theta, \theta\in[0,\pi)$&with $\alpha=\frac{\cos\theta}{\sin\theta}\geq0$, when $\theta\neq0$\\ \hline  \hline
\multirow{3}{*}{\text{(Sii)}}&$d\varphi^1=0,\quad d\varphi^2=-\frac{1}{2}\varphi^{13}-\left(\frac{1}{2}+\sqrt{-1} x\right) \varphi^{1\bar3}+\sqrt{-1}x\,\varphi^{3\bar1},$&\multirow{3}{*}{$\mathfrak{g}_3=(0, -13, 12, 0, -46, -45)$} \\
&$d\varphi^3=\frac{1}{2}\varphi^{12}+\left(\frac{1}{2}-\frac{\sqrt{-1}}{4x}\right)\varphi^{1\bar2}+\frac{\sqrt{-1}}{4x}\varphi^{2\bar1}$,&  \\
&where $x\in\mathbb R^{>0}$&\\ \hline \hline
\multirow{3}{*}{\text{(Siii1)}}&$d\varphi^1=\sqrt{-1}\varphi^{13}+\sqrt{-1}\varphi^{1\bar3}$&\multirow{3}{*}{$\mathfrak{g}_4=(23, -36, 26, -56, 46, 0)$} \\
&$d\varphi^2=-\sqrt{-1}\varphi^{23}-\sqrt{-1}\varphi^{2\bar3}$&  \\
&$d\varphi^3=\pm \varphi^{1\bar1}$&\\ \hline
\multirow{3}{*}{\text{(Siii2)}}&$d\varphi^1=\varphi^{13}+\varphi^{1\bar3}$&\multirow{3}{*}{$\mathfrak{g}_5=(24 + 35, 26, 36, -46, -56, 0)$} \\
&$d\varphi^2=-\varphi^{23}-\varphi^{2\bar3}$&  \\
&$d\varphi^3=\varphi^{1\bar2}+\varphi^{2\bar1}$&\\ \hline
\multirow{3}{*}{\text{(Siii3)}}&$d\varphi^1=\sqrt{-1}\varphi^{13}+\sqrt{-1}\varphi^{1\bar3}$&\multirow{3}{*}{$\mathfrak{g}_6=(24 + 35, -36, 26, -	56, 46, 0)$} \\
&$d\varphi^2=-\sqrt{-1}\varphi^{23}-\sqrt{-1}\varphi^{2\bar3}$&  \\
&$d\varphi^3=\varphi^{1\bar1}+\varphi^{2\bar2}$&\\ \hline
\multirow{3}{*}{\text{(Siii4)}}&$d\varphi^1=\sqrt{-1}\varphi^{13}+\sqrt{-1}\varphi^{1\bar3}$&\multirow{3}{*}{$\mathfrak{g}_7=(24 + 35, 46, 56, -26, -36, 0)$} \\
&$d\varphi^2=-\sqrt{-1}\varphi^{23}-\sqrt{-1}\varphi^{2\bar3}$&  \\
&$d\varphi^3=\pm(\varphi^{1\bar1}-\varphi^{2\bar2})$&\\ \hline  \hline
\text{(Siv1)}&$d\varphi^1=-\varphi^{13}, \quad d\varphi^2=\varphi^{23},\quad d\varphi^3=0$&\multirow{6}{*}{$\mathfrak{g}_8=(16-25, 15+26, -36+45, -35-46, 0, 0)$} \\ \cline{1-2}
\multirow{2}{*}{\text{(Siv2)}}&$d\varphi^1=2\sqrt{-1}\varphi^{13}+\varphi^{3\bar3},\quad x \in \{0,1\}$ & \\ &$d\varphi^2=-2\sqrt{-1}\varphi^{23}+x\,\varphi^{3\bar3}, \quad d\varphi^3=0$&\\ \cline{1-2}
\multirow{3}{*}{\text{(Siv3)}}&$d\varphi^1=A\,\varphi^{13}-\varphi^{1\bar3}$ & \\ &$d\varphi^2=-A\,\varphi^{23}+\varphi^{2\bar3}, \quad d\varphi^3=0$&\\
&$ A \in \mathbb C \text{ with } |A| \neq 1$&\\ \hline \hline
\multirow{3}{*}{\text{(Sv)}}&$d\varphi^1=-\varphi^{3\bar3}$&\multirow{3}{*}{$\mathfrak{g}_9=(45, 15 + 36, 14 - 26 + 56, -56, 46, 0)$} \\
&$d\varphi^2=\frac{\sqrt{-1}}{2}\varphi^{12}+\frac{1}{2}\varphi^{1\bar3}-\frac{\sqrt{-1}}{2}\varphi^{2\bar1}$&  \\
&$d\varphi^3=-\frac{\sqrt{-1}}{2}\varphi^{13}+\frac{\sqrt{-1}}{2}\varphi^{3\bar1}$&\\ \hline \hline
 \end{tabular}
 }}
 \caption{Invariant complex structures on six-dimensional solvmanifolds non-nilmanifolds with holomorphically-trivial canonical bundle up to linear equivalence, see \cite{otal-thesis, fino-otal-ugarte}.}
 \label{table:solv-cplx}
\end{center}
\end{table}

In particular, notice that the family (Np) and the family (Siv1) consist of {\em holomorphically-parallelizable} structures, namely, such that the holomorphic tangent bundle is holomorphically-trivial.
The case (Np) with $\rho=0$ corresponds to the complex torus, the case (Np) with $\rho=1$ corresponds to the holomorphically-parallelizable Iwasawa manifold, and the case (Siv1) corresponds to the holomorphically-parallelizable Nakamura manifold.

We recall that, by \cite[Theorem 1]{wang}, holomorphically-parallelizable manifolds can be regarded, up to a holomorphic
homeomorphism, as quotients of a connected simply-connected complex Lie group by a discrete subgroup.
Clearly, any invariant metric on a holomorphically-parallelizable manifold is Chern-flat. (On the other side, by \cite{boothby}, a compact complex Hermitian manifold is Chern-flat if and only if its universal cover is holomorphically isometric to a complex Lie group endowed with an invariant Hermitian metric.)
In \cite[Theorem 1.2]{yang-zheng}, the authors prove that, on a compact complex manifold of complex dimension $n\geq3$, a metric such that both its Levi-Civita and its Chern connection are K\"ahler-like is actually K\"ahler. Recall that, by \cite[Corollary 2]{wang}, holomorphically-parallelizable compact K\"ahler manifolds are complex torus.
Then, on holomorphically-parallelizable nilmanifolds and solvmanifolds different than tori, the Chern connection associated to an invariant Hermitian structure is always K\"ahler-like, while the Levi-Civita connection is never.

Note that this case also includes the case of the {\em special Lie algebra $\mathfrak{sl}(2;\mathbb C)$}, with structure equations
\begin{equation}\tag{sl2C}\label{eq:sl2C}
d\varphi^1=\varphi^{23}, \quad d\varphi^2=-\varphi^{13}, \quad d\varphi^3=\varphi^{12} ,
\end{equation}
other than the already mentioned  {\em Iwasawa manifold} (Np) with $\rho=1$,
and {\em Nakamura manifold} (Siv1).

\subsection{Special Hermitian metrics}
In view of \cite[Theorem 1.3]{yang-zheng} and Conjectures \ref{conj:bismut-skt} and \ref{conj:gauduchon-kahler-like}, we are particularly interested on nilmanifolds and solvmanifolds admitting invariant complex structures with holomorphically-trivial canonical bundle and (invariant) balanced or pluriclosed metrics.

As for balanced, according to \cite{ugarte-transf}, and \cite{couv, ugarte-villacampa-Asian, ugarte-villacampa}, they are:
\begin{itemize}
\item either $\mathfrak{h}_1$ with a holomorphically-parallelizable complex structure in Family (Np) with $\rho=0$; any Hermitian metric is in fact K\"ahler;
\item or $\mathfrak{h}_2$, \dots, $\mathfrak{h}_6$ with a nilpotent complex structure in Family (Ni); a generic Hermitian metric as in \eqref{eq:metric} is balanced if and only if $r^2 = 1$, $v = z = 0$, and $s^2 + D = \sqrt{-1}\, \bar u\, \lambda$;
\item or $\mathfrak{h}_5$ with a holomorphically-parallelizable complex structure in Family (Np) with $\rho=1$, corresponding to the Iwasawa manifold; in this case, any invariant Hermitian metric is balanced;
\item or $\mathfrak{h}_{19}^{-}$ with a non-nilpotent complex structure in Family (Niii); balanced metrics in \eqref{eq:metric} are characterized by $u=z=0$, and either $t^2=1$ and $v=0$, or $t^2>0$ and $v=1$;
\item or $\mathfrak g_1$ or $\mathfrak g_2^{\alpha}$ with a splitting-type complex structure in Family (Si); a generic metric as in \eqref{eq:metric} is balanced if and only if $v=z=0$; they are non-K\"ahler except for $\mathfrak g_2^{0}$ with $u=v=z=0$;
\item or $\mathfrak{g}_3$ (Sii) or $\mathfrak{g}_5$ (Siii2) or $\mathfrak{g}_7$ (Siii4); balanced metrics in \eqref{eq:metric} are characterized by $v=z=0$, and moreover in case $\mathfrak{g}_5$ take $u\in\mathbb R$, respectively in case $\mathfrak g_7$ take $r^2=s^2$;
\item or $\mathfrak g_8$ with a holomorphically-parallelizable splitting-type complex structure in Family (Siv1): any metric is balanced; or with a splitting-type complex structure in Family (Siv3): balanced metrics as in \eqref{eq:metric} are characterized by $v=z=0$.
\end{itemize}

As for pluriclosed metrics, they exist only on:
\begin{itemize}
\item $\mathfrak{h}_1$ with a holomorphically-parallelizable complex structure in Family (Np) with $\rho=0$; any Hermitian metric is in fact K\"ahler;
\item $\mathfrak h_2$, $\mathfrak h_4$, $\mathfrak h_5$, $\mathfrak h_8$ in Family (Ni), and in this case any Hermitian metric is pluriclosed, see \cite[Theorems 1.2, 3.2]{fino-parton-salamon};
\item $\mathfrak g_2^0$ in Family (Si) and $\mathfrak g_4$ (Siii1); pluriclosed metrics are given by $u=0$ in the form \eqref{eq:metric}; see \cite[Theorem 4.1]{fino-otal-ugarte}.
\end{itemize}

Finally, we review the existence of K\"ahler metrics. By \cite{benson-gordon} and, more in general, \cite{hasegawa}, K\"ahler metrics do not exist on non-tori nilmanifolds (with invariant or non-invariant complex structures). On the other side, K\"ahler metrics exist on:
\begin{itemize}
\item $\mathfrak h_1$, namely, the complex torus;
\item $\mathfrak g_2^0$ in Family (Si), and the K\"ahler metrics are given by the diagonal ones, $u=v=z=0$, see \cite[Theorem 5.1.3]{otal-thesis}.
\end{itemize}

\begin{rmk}
Observe that the Lie algebras $\frh_7,\ldots,\frh_{16},\frh_{26}^+$ and $\mathfrak{g}_9$ do not admit neither balanced nor pluriclosed metrics.
\end{rmk}

\section{K\"ahler-like Gauduchon connections on six-dimensional Calabi-Yau solvmanifolds}

In this section we study the existence of K\"ahler-like connections on the class of 6-dimensional Calabi-Yau solvmanifolds.

Let us denote the Gauduchon connections by $ \nabla^{\varepsilon}:=\nabla^{G_{1-4\varepsilon}}$ and by $R^\varepsilon$ their
corresponding curvature. We recall that with respect to this notation, $\nabla^{\varepsilon=0} = \nabla^{Ch}$  and $\nabla^{\varepsilon = \sfrac{1}{2}} = \nabla^+$.

\begin{thm}\label{thm:solv}
Let $X=\Gamma \backslash G$ be a six-dimensional solvmanifold endowed with an invariant complex structure $J$ such that the canonical bundle is holomorphically-trivial.
Denote by $\mathfrak{g}$ the Lie algebra associated to $G$ and let $\omega$ be the $(1,1)$-form associated to any invariant $J$-Hermitian metric on $X$. We have the following.
\begin{itemize}
\item If the Chern connection $\nabla^{Ch}$ is K\"ahler-like, then $\mathfrak{g}$ is isomorphic to $\mathfrak{h}_1$, $\mathfrak{h}_5$, $\mathfrak{g}_1$, $\mathfrak{g}_2^{\alpha\geq0}$, or $\mathfrak{g}_8$, and the Hermitian metric $\omega$ is balanced.
\item If the Strominger-Bismut connection $\nabla^+$ is K\"ahler-like, then $\mathfrak{g}$ is isomorphic to $\mathfrak{h}_1$, $\mathfrak{h}_2$, $\mathfrak{h}_8$,
$\mathfrak{g}_2^{0}$, or $\mathfrak{g}_4$, and the Hermitian metric $\omega$ is pluriclosed.
\item If a Gauduchon connection $\nabla^\varepsilon$ is K\"ahler-like for some $\varepsilon \in \mathbb{R}\setminus\{0,\frac12\}$,
then $\mathfrak{g}$ is isomorphic to $\mathfrak{h}_1$ or $\mathfrak{g}_2^{0}$, and the Hermitian metric $\omega$ is K\"ahler.
\end{itemize}
\end{thm}

Notice that Theorem \ref{thm:solv} yields that Conjectures~\ref{conj:bismut-skt} and~\ref{conj:gauduchon-kahler-like} are satisfied for such class of six-dimensional Calabi-Yau solvmanifolds.
It also shows that the property of being K\"ahler-like for the Chern connection is neither open nor closed under holomorphic deformations of the complex structure: as for non-closedness, it follows by

\begin{cor}\label{cor:kl-non-closed}
The Lie algebra $\mathfrak g_8$ admits Chern-flat Hermitian metrics for complex structures in Family (Siv3), which admit limits in Family (Siv1) which are not even balanced \cite[Theorem 5.2]{fino-otal-ugarte}.
\end{cor}

The following result makes Theorem~\ref{thm:solv} more precise, in specifying the complex structures $J$ and the $J$-Hermitian metrics $\omega$ in each case when the connections are K\"ahler-like.

\begin{thm}\label{thm:solv-precise}
In the conditions of Theorem~\ref{thm:solv}, we have the following.
\begin{itemize}
\item The Chern connection $\nabla^{Ch}$ is K\"ahler-like precisely in the following cases:

\begin{itemize}
\item $\mathfrak{h}_1$, with $J$ in Family (Np) with $\rho=0$, with any $\omega$ given by \eqref{eq:metric}, which is in fact K\"ahler and Chern-flat;

\item $\mathfrak{h}_5$, with $J$ in the holomorphically-parallelizable Family (Np) with $\rho=1$, with any $\omega$ given by \eqref{eq:metric}, which is in fact Chern-flat;

\item $\mathfrak{g}_1$, with $J$ in Family (Si) with $\theta=0$, with $\omega	$ given by \eqref{eq:metric} with $u=v=z=0$, which is in fact Chern-flat;

\item $\mathfrak{g}_2^0$, with $J$ in Family (Si) with $\theta=\frac{\pi}{2}$,  with $\omega$ given by \eqref{eq:metric} with $u=v=z=0$, which is in fact K\"ahler and Chern-flat;

\item $\mathfrak{g}_2^{\alpha}$ with $\alpha>0$, with $J$ in Family (Si) with $\theta\not\in\{0,\frac{\pi}{2}\}$, with $\omega$ given by \eqref{eq:metric} with $u=v=z=0$, which is in fact Chern-flat;

\item $\mathfrak{g}_8$, with $J$ in the holomorphically-parallelizable Family (Siv1), with any $\omega$ given by \eqref{eq:metric}, which is in fact Chern-flat;

\item $\mathfrak{g}_8$, with $J$ in Family (Siv3), with $\omega$ given by \eqref{eq:metric} with $u=v=z=0$, which is in fact Chern-flat.
\end{itemize}
In any case, the metric is balanced by \cite[Theorem 1.3]{yang-zheng}.

\item The Strominger-Bismut connection $\nabla^+$ is K\"ahler-like precisely in the previous K\"ahler cases (see $\mathfrak h_1$ and $\mathfrak g_2^0$) and in the following cases:

\begin{itemize}
\item $\mathfrak{h}_2$, with $J$ in Family (Ni) with $\rho=\lambda=0$ and $D=\sqrt{-1}$, with $\omega$ given by \eqref{eq:metric} with $r^2=1$, $u=v=z=0$ (not Strominger-Bismut-flat);

\item $\mathfrak{h}_8$, with $J$ in Family (Ni) with $\rho=\lambda=D=0$,  with any $\omega$ given by \eqref{eq:metric} (not Strominger-Bismut-flat);

\item $\mathfrak{g}_4$, with $J$ in Family (Siii1), with $\omega$ given by \eqref{eq:metric} with $u=v=z=0$ (not Strominger-Bismut-flat).
\end{itemize}
In any case, the metric is pluriclosed.

\item The Gauduchon connection $\nabla^{\varepsilon}$ for some $\varepsilon \in \mathbb{R}\setminus\{0,\frac12\}$ is K\"ahler-like precisely in the previous K\"ahler cases (see $\mathfrak h_1$ and $\mathfrak g_2^0$).
\end{itemize}
\end{thm}

The rest of the section is devoted to the proofs of Theorems~\ref{thm:solv} and~\ref{thm:solv-precise}.
Recall that by Remark~\ref{caracterizacionKL} we know that $\nabla^{\varepsilon}$ is K\"ahler-like if and only if the
identities \eqref{eq:bianchi1-v2} and \eqref{eq:type-v2} hold;
more concretely,
$$
R^\varepsilon_{ij\bullet\bullet}=R^\varepsilon_{\bullet\bullet k\ell}=0,\text{ and }
B^\varepsilon_{i\bar jk \bar\ell}:= R^\varepsilon_{i\bar j k\bar \ell}- R^\varepsilon_{k\bar j i\bar \ell}=0,
$$
for any $i,j,k,\ell\in\{1,2,3\}$, where for instance $B^\varepsilon_{1\bar 2 3\bar 1}$ denotes $B^\varepsilon(X_1,\bar X_2,X_3,\bar X_1)$.
Throughout this section,
and also Section~\ref{LCKL}, we make use of this notation.

\begin{rmk}\label{rmk:non-cpt}
We remark that we are doing calculations at the level of the Lie algebra, equivalently, on invariant objects on the Lie group. In particular, this means that our results hold true also at the level of the {\em non-compact} Lie group.
Compare also \cite[Theorem 1.3]{vezzoni-yang-zheng}, where no compactness assumptions is supposed, but the extra assumption on invariant parallel frame, to study flat Gauduchon connections on Lie groups.
\end{rmk}

\subsection{Nilmanifolds}
\subsubsection{Holomorphically-parallelizable nilmanifolds in Family (Np)}
Consider the complex structure equations
$$ d\varphi^1=d\varphi^2=0,\quad d\varphi^3 = \varphi^{12},$$
and a generic (balanced) Hermitian metric given by \eqref{eq:metric}. Let $\{X_1,\,X_2,X_3\}$ be the (1,0)-basis of vectors dual to
$\{\varphi^1,\,\varphi^2,\, \varphi^3\}$, that is, $[X_1,X_2]=-X_3$.

\begin{prop}
The Gauduchon connection $\nabla^{\varepsilon}$ is never K\"ahler-like unless $\varepsilon = 0$.
\end{prop}

\begin{proof}
A direct calculation shows that $R^{\varepsilon}_{ij\bullet\bullet}=R^{\varepsilon}_{\bullet\bullet k\ell}=0$ for any $i,j,k,\ell\in\{1,2,3\}$.  On the other hand, what symmetries of Bianchi-type concerns, we can observe that:
$$
B^{\varepsilon}_{1\bar{1}3\bar{3}} = \frac{2\varepsilon^2t^4(r^2t^2-|z|^2)}{8\sqrt{-1}\det\Omega}= 0\Longleftrightarrow \varepsilon = 0.
$$

Moreover, $B^{0}\equiv 0$, {\itshape i.e.} $\nabla^{Ch}$ is K\"ahler-like; in fact, this is true for any holomorphically-parallelizable manifold.
\end{proof}

\subsubsection{Nilmanifolds in Family (Ni)}
According to \cite[equations (2.4)--(2.5)]{ugarte-villacampa}, any Hermitian structure can be expressed as:

\begin{gather}\label{F-I}
d\varphi^1=d\varphi^2=0,\quad d\varphi^3 = \rho\,\varphi^{12} + \varphi^{1\bar 1} + \lambda\,\varphi^{1\bar2} +  D\,\varphi^{2\bar2}, \\
2\omega = \sqrt{-1}(\varphi^{1\bar 1} + s^2\,\varphi^{2\bar2} +  t^2\,\varphi^{3\bar3}) + u\,\varphi^{1\bar2} -\bar u\,\varphi^{2\bar1},\nonumber
\end{gather}
where $\rho\in\{0,1\},\, \lambda\geq 0,\, \Im D\geq 0,$ and $s^2>|u|^2,\, t^2>0$; {\itshape i.e.} we can take $v=z=0$ and $r^2=1$ in the generic expression \eqref{eq:metric}.

\begin{lem}
In the notation as above, if $\nabla^{\varepsilon}$ is K\"ahler-like, then $\varepsilon=\frac12$, $\rho=0$ and $\omega$ is pluriclosed.
\end{lem}

\begin{proof}
Let us take the element   $R^{\varepsilon}_{231\bar{3}}=\frac{2 \varepsilon^2 \rho s^2 t^6}{8\sqrt{-1}\det\Omega}$. It vanishes if and only
if $\varepsilon\rho = 0$. If $\varepsilon=0$, then $B^{0}_{1\bar{1}3\bar{3}} = -\frac{s^2 t^4}{2(s^2-|u|^2)}\neq0$, hence $\rho=0$.

Now, substituting $\rho=0$ we find:
$$B^{\varepsilon}_{1\bar{1}2\bar{1}}=\frac{-\lambda (1-2\varepsilon)^2t^2}{2}=0 \Longleftrightarrow \varepsilon=\frac12\text{ or } \lambda = 0.$$
Finally, if $\lambda=0$, then $B^{\varepsilon}_{1\bar{3}3\bar{1}} = \frac{s^2 t^4\varepsilon  (2\varepsilon-1) }{s^2-|u|^2}\neq 0$, hence $\varepsilon=\frac12$.

For the last statement, according to \cite[equation (3)]{fino-parton-salamon}, the Hermitian metric $\omega$ is pluriclosed if and only if the parameters
in the complex structure (Ni) satisfy
$$\rho + \lambda^2 -(D+\bar D) = 0.$$
Now, in the case $\rho=0$, we have that:
$$B^{\frac12}_{1\bar{1}2\bar{2}} = \frac{-t^2}{2} (\lambda^2 - (D+\bar D)),$$
whence the statement.
\end{proof}

\begin{prop}
In the notation as above, $\nabla^{\varepsilon}$ is K\"ahler-like if and only if $\varepsilon=\frac12$ and

\begin{itemize}
\item either the Lie algebra is $\frh_2$ and the Hermitian structure is given by   $(\rho, \lambda, D) = (0,0,\sqrt{-1})$ and $u=0$,
\item or the Lie algebra is $\frh_8$ and the Hermitian structure is anyone defined on  the complex structure $(\rho, \lambda, D) = (0,0,0)$.
\end{itemize}

\end{prop}

\begin{proof}
  In \cite{ugarte-transf}, it is shown that the only Lie algebras underlying equations~\eqref{F-I} with $\rho=0$ admitting a Hermitian pluriclosed metric are $\frh_2$ and $\frh_8$.  Moreover, $\Re D = \frac{\lambda^2}{2}$.  Taking into account the classification of complex structures up to equivalence \cite{couv}, $\lambda$ can take only the value 0, so we are forced to:
  \begin{itemize}
  \item $(\rho, \lambda, D) = (0,0,\sqrt{-1})$, namely, the Lie algebra $\frh_2$;
  \item $(\rho, \lambda, D) = (0,0,0)$, namely, the Lie algebra $\frh_8$.
  \end{itemize}
  We are going to study the two cases above.  First, the computation of the curvature elements for $\rho=0$ and $\varepsilon = \sfrac{1}{2}$ yields that, for any $i,j,k,\ell\in\{1,2,3\}$,
  $$R^{\frac12}_{ij\bullet\bullet}=R^{\frac12}_{\bullet\bullet k\ell}=0.$$
  With respect to the relations coming from the Bianchi identity, we summarize the results of the computations in Appendix \ref{app:BC-Ni}.
  In particular, $\varepsilon=\frac12$, $\rho=0$, $\lambda=0$, and $\Re D=0$ yield that
  \begin{equation*}\label{SKT}
  B^{\frac12}\equiv 0\Longleftrightarrow B^{\frac12}_{1\bar{3}3\bar{2}}=\frac{ t^4 u (D-\bar D)}{16\det\Omega}=0 \Longleftrightarrow  u\, \Im D=0.
  \end{equation*}
\end{proof}

\begin{rmk}
In the cases above, the Strominger-Bismut connection is not flat.
In fact, $R^+_{1\bar{1}1\bar{1}} = R^+_{2\bar{2}2\bar{2}} = t^2\neq 0$ for $\frak h_2$,
and $R^+_{1\bar{1}1\bar{1}} = t^2\neq 0$ for $\frak h_8$, the other components of the curvature being zero.
\end{rmk}

\subsubsection{Nilmanifolds in Family (Nii)}
Consider the complex structure equations
$$ d\varphi^1=0,\quad d\varphi^2=\varphi^{1\bar1},\quad d\varphi^3 = \rho\,\varphi^{12} +B \varphi^{1\bar 2} +c \,\varphi^{2\bar1}, $$
where $\rho\in\{0,1\},\, c\geq 0,\, B\in \mathbb C$ satisfying $(\rho, B, c)\neq (0,0,0)$.

\begin{prop}
For nilmanifolds in Family (Nii), the Gauduchon connection $\nabla^{\varepsilon}$ is never K\"ahler-like.
\end{prop}

\begin{proof}
For the Chern connection, i.e. $\varepsilon=0$, the result follows directly from the fact that $\omega$ is never balanced and by \cite[Theorem 1.3]{yang-zheng}.

As a consequence, we consider $\varepsilon\neq 0$ and the following elements:
$$
\begin{array}{lcl}
  B^{\varepsilon}_{3\bar{2}2\bar{3}} &=& -\frac{t^4 (s^2t^2-|v|^2)}{4\sqrt{-1}\,\det\Omega} \left(\varepsilon^2 \rho^2 - 2|B|^2 \left(\varepsilon-\frac12\right)\left(\varepsilon-\frac14\right)\right),\\[1em]
  B^{\varepsilon}_{3\bar{3}2\bar{2}} &=& \frac{\varepsilon\, t^4 (s^2t^2-|v|^2)}{4\sqrt{-1}\,\det\Omega} \left(\varepsilon c^2 -2 |B|^2 \left(\varepsilon-\frac14\right)\right),

\end{array}
$$
where recall the expression for $\det\Omega$ as in \eqref{det}.

Observe first that $R^{\varepsilon}_{232\bar{3}}=\frac{2 \varepsilon^2 \rho \bar B (s^2 t^2-|v|^2)t^4}{8\sqrt{-1}\det\Omega}=0$ if and only if
$\rho B = 0$.  If $B=0$, then $$(B^{\varepsilon}_{3\bar{2}2\bar{3}}, B^{\varepsilon}_{3\bar{3}2\bar{2}}) = (0,0)\Longleftrightarrow (\rho, c)=(0,0)$$ which is a contradiction.  Therefore, we may assume $B\neq 0$ and $\rho=0$.  Now,
$$B^{\varepsilon}_{3\bar{2}2\bar{3}}=0 \Longleftrightarrow \varepsilon = \frac12 \text{ or } \varepsilon = \frac14.$$

If $\varepsilon = \frac14$, then
$$B^{\frac14}_{3\bar{3}2\bar{2}}=0 \Longleftrightarrow c=0$$
but in this case one can check that $B_{2\bar{1}1\bar{2}}^{\frac14} = \frac{|B|^2 t^2}{4}\neq 0$.

If $\varepsilon = \frac12$, then $R^{\frac12}_{231\bar{2}}=\frac{c (s^2 t^2 - |v|^2)^2}{16\sqrt{-1}\det\Omega}$ and $R^{\frac12}_{232\bar{1}}=\frac{-\bar B (s^2 t^2 - |v|^2)^2}{16\sqrt{-1}\det\Omega}$, therefore
$$ (R^{\frac12}_{231\bar{2}}, R^{\frac12}_{232\bar{1}}) = (0,0)\Longleftrightarrow (B, c)=(0,0) ,$$
providing a contradiction.
\end{proof}

\subsubsection{Nilmanifolds in Family (Niii)}
Consider the complex structure equations
$$d\varphi^1=0,\quad d\varphi^2=\varphi^{13}+\varphi^{1\bar3},\quad d\varphi^3 =\sqrt{-1}\nu\,\varphi^{1\bar 1}+ \sqrt{-1}\,\delta\, (\varphi^{1\bar 2} -\varphi^{2\bar1}), $$
where $\nu = \{0,1\}$ and $\delta = \pm 1$.

\begin{prop}
For nilmanifolds in Family (Niii), the Gauduchon connection $\nabla^{\varepsilon}$ is never K\"ahler-like.
\end{prop}

\begin{proof}
Consider the Bianchi identity relation $B^{\varepsilon}_{3\bar{3}2\bar{3}}$.
Now,
\begin{eqnarray*}
B^{\varepsilon}_{3\bar{3}2\bar{3}}&=&\frac{(4 \varepsilon^2+4 \varepsilon-1)(-\sqrt{-1} s^2+\delta t^2) (s^2 t^2-|v|^2)v}{16\sqrt{-1}\det\Omega}=0 \\
&\Longleftrightarrow& \begin{cases}
v=0,\\
\varepsilon = \frac{-1\pm  \sqrt 2}{2}.
\end{cases}
\end{eqnarray*}

We start supposing $v=0$.  Now, $B^{\varepsilon}_{3\bar{2}1\bar{2}} = \frac{2 \varepsilon^2 \bar z (\sqrt{-1} s^2 +\delta t^2)s^4 }{8\sqrt{-1}\det\Omega}=0$ if and only
if $z\varepsilon = 0$.  In the case $\varepsilon =0$, one can check that $B^0_{3\bar{3}2\bar{2}}=\frac{-s^4 t^2 (s^2 +\delta \sqrt{-1} t^2)}{16\sqrt{-1}\det\Omega}$ is always a
non-zero element.  On the other hand, if $z=0$, just take
$$
B^{\varepsilon}_{2\bar{2}3\bar{3}}=\frac{s^2t^2(s^2+ \sqrt{-1} \delta t^2)
\left(2\varepsilon s^2+\sqrt{-1}(1-2\varepsilon)\delta t^2\right)(4 \varepsilon-1)}{16\sqrt{-1}\det\Omega}.
$$

One has $B^{\varepsilon}_{2\bar{2}3\bar{3}}=0 \Longleftrightarrow \varepsilon = \frac14$.
But in this case, $B^{\frac14}_{1\bar{1}3\bar{3}}=\frac{-s^2}{2}\neq 0$.

Finally, if $\varepsilon_0=\frac{-1\pm  \sqrt 2}{2}$ is a root of the polynomial $4\varepsilon^2+4\varepsilon-1$, then:
\begin{eqnarray*}
\lefteqn{ B^{\varepsilon_0}_{3\bar{3}2\bar{2}}=0}\\
&\Longleftrightarrow& (s^2t^2-|v|^2)\left((s^2 + \sqrt{-1}\,\delta t^2) \left[(\mp 7 +5\sqrt 2) s^2 \pm \sqrt{-1} t^2 (4\pm 3\sqrt2)\right] \right. \\
&& \left. -\sqrt{-1}\,\delta |v|^2 (\mp 3+2\sqrt 2)\right)=0,
\end{eqnarray*}
but the last expression is always non-zero.
\end{proof}

\subsection{Solvmanifolds}
\subsubsection{Solvmanifolds in Family (Si)}
By Section~\ref{subsec:complex-st}, for this family of complex structures the underlying Lie algebras are $\frak g_1$ or $\frak g_2^{\alpha}$ with $\alpha\geq 0$.

\begin{prop}
The Chern connection $\nabla^{Ch}$ is K\"ahler-like if and only $\omega$ given by \eqref{eq:metric} satisfies $u=v=z=0$.
Moreover, in these cases, the Chern curvature vanishes identically.
\end{prop}

\begin{proof}
Using \cite[Theorem 1.3]{yang-zheng}, a necessary condition for $\nabla^{Ch}$ to be K\"ahler-like is that $\omega$ must be a balanced metric.
According to \cite[Theorem 4.5]{fino-otal-ugarte}, $\omega$ is balanced if and only if $v=z=0$.
Now, with this restriction, it is immediate to see that $B^0\equiv 0$ if and only if $u=0$ (see Appendix \ref{sec:app:g2a-Chern}).
The last statement follows by direct inspection.
\end{proof}

\begin{rmk}\label{rmk:g12}
The previous result gives an example of a non-K\"ahler metric with K\"ahler-like Chern connection on a solvmanifold endowed with a left-invariant complex structure which does not admit any basis of (left-invariant) holomorphic vector fields.

Note  indeed  that  the  above  diagonal  metrics  are  K\"ahler if
and only if the  parameter  satisfies $A=\sqrt{-1}$, corresponding to the Lie algebra $\mathfrak g^0_2$.
Clearly, by \cite{boothby}, the complex solvmanifold is in fact biholomorphic to a holomorphically-parallelizable manifold. But it is presented in a different way as quotient of a Lie group, and as such it is not parallelizable by left-invariant holomorphic vector fields.
\end{rmk}

\begin{prop}
For $\varepsilon\neq 0$, the Gauduchon connection
$\nabla^{\varepsilon}$ is K\"ahler-like if and only if $\frak g= \frak g_2^0$ and $\omega$ is K\"ahler.
\end{prop}

\begin{proof}
  We prove firstly that the underlying Lie algebra must be $\frak g_2^0$, corresponding to $A=\sqrt{-1}$ in family (Si), as:
  $$B^{\varepsilon}_{1\bar{1}2\bar{2}} = \frac{8\,\varepsilon^2(r^2s^2-|u|^2)}{8\sqrt{-1}\, \det\Omega} \left[(\Re A)^2 (r^2s^2-|u|^2) + |A|^2|u|^2\right]=0,$$
  implies $A=\sqrt{-1}$ and $u=0$.

  Now, as it is shown in Appendix \ref{sec:app:g2a}, $B^{\varepsilon}_{1\bar{1}3\bar{3}} = B^{\varepsilon}_{2\bar{2}3\bar{3}} = 0$
  if and only if $(2\varepsilon-1)v = (2\varepsilon-1)z=0$.

  If $v=z=0$, in particular the metric $\omega$ is K\"ahler and one has that $B^{\varepsilon}\equiv 0$ and $R^{\varepsilon}\equiv 0$.

  On the other hand, consider now the Strominger-Bismut connection, {\itshape i.e.} $\varepsilon = \frac12$.
  Then, $B^{\frac12}\equiv 0$ if and only if
  $(B^{\frac12}_{1\bar{3}3\bar{3}}, B^{\frac12}_{2\bar{3}3\bar{3}}) = (\frac{\sqrt{-1}z}{2}, \frac{\sqrt{-1}v}{2}) = (0,0)$,
  which implies $v=z=0$.
\end{proof}

\subsubsection{Solvmanifolds in Family (Sii)}
By Section~\ref{subsec:complex-st}, the Lie algebra underlying this family of complex structures is $\frak g_3$.

\begin{prop}
For any solvmanifold in Family (Sii), the Gauduchon connection $\nabla^{\varepsilon}$ is never K\"ahler-like.
\end{prop}

\begin{proof}
First let us suppose $\varepsilon=0$. Using \cite[Theorem 1.3]{yang-zheng}, a necessary condition for $\nabla^{Ch}$ to be K\"ahler-like is that $\omega$ must be a balanced metric.  According to \cite[Theorem 4.5]{fino-otal-ugarte}, $\omega$ is balanced if and only if $u=z=0$.
Now, with this restriction,
$$
B^0_{1\bar{2}2\bar{1}} = \frac{-(t^2 - 2 \sqrt{-1} s^2 x)(1+4x^2)}{32 x^2}\neq 0,
$$
whence the statement.

In what follows, we will suppose that $\varepsilon\neq 0$ and a generic Hermitian metric $\omega$.
Consider the elements $R^{\varepsilon}_{232\bar{3}}$ and $R^{\varepsilon}_{233\bar{3}}$:
$$
\begin{cases}
R^{\varepsilon}_{232\bar{3}}=\frac{\varepsilon (\sqrt{-1} + 2 x) (s^2 t^2 - |v|^2) (-v^2 + 2 \sqrt{-1} s^2 t^2 x -
   4 s^4 x^2 + 2 \varepsilon (t^4 + 4 s^4 x^2) - 2 \sqrt{-1} |v|^2 x)}{64x\sqrt{-1}\det\Omega},\\[5pt]
R^{\varepsilon}_{233\bar{3}}=\frac{\varepsilon (\sqrt{-1} + 2 x) (s^2 t^2 -
   |v|^2) (v (\sqrt{-1} (-1 + 2 \varepsilon) t^2 + 4 \varepsilon s^2 x) +
   4 \sqrt{-1} x (\sqrt{-1} \varepsilon t^2 - s^2 x + 2 \varepsilon s^2 x) \bar v)}{64x\sqrt{-1}\det\Omega}.
\end{cases}
$$

Now:
$$
\begin{cases}
R^{\varepsilon}_{232\bar{3}}=0\Longleftrightarrow 2\varepsilon t^4 - v^2 + 2\sqrt{-1}x (s^2t^2 - |v|^2) + 4s^4x^2(2\varepsilon-1)=0,\\[5pt]
R^{\varepsilon}_{233\bar{3}}=0\Longleftrightarrow v\left(4\sqrt{-1}\,\varepsilon s^2 x - (2\varepsilon-1) t^2\right) - 4x\bar v\left(\sqrt{-1}\,\varepsilon t^2 + (2\varepsilon - 1)s^2 x\right)=0.
\end{cases}
$$
Adding and subtracting the two expressions above, we obtain the equivalent system:
$$
\begin{cases}
(2\sqrt{-1}\,s^2x + 2\varepsilon (t^2-2\sqrt{-1}\,s^2 x) - v)(t^2+2\sqrt{-1}\,s^2 x + v + 2\sqrt{-1}\,x\bar v)=0,\\[5pt]
(2\sqrt{-1}\,s^2x + 2\varepsilon (t^2-2\sqrt{-1}\,s^2 x) + v)(t^2+2\sqrt{-1}\,s^2 x - v - 2\sqrt{-1}\,x\bar v)=0.
\end{cases}
$$
As a matter of notation, let us write the first equation as $AB=0$ and the second as $CD=0$.

Let us focus our attention in the first equation and consider the case $A=0$.  We can express $v$ as:
$$v = 2\sqrt{-1}\,s^2x + 2\varepsilon (t^2-2\sqrt{-1}\,s^2 x).$$
Substituting this value in the second equation we obtain that:
$$
\begin{cases}
C=4(\varepsilon t^2 - \sqrt{-1}s^2 x (2\varepsilon -1))\neq 0,\\[5pt]
D = (1-2\varepsilon) (t^2-4s^2x^2) + 4\sqrt{-1}\,\varepsilon x (s^2-t^2)=0\Longleftrightarrow \begin{cases}
\Re D=0,\\ \Im D=0,
\end{cases}
\end{cases}
$$
Now,
$$
\begin{cases}
\Re D=0,\\ \Im D=0,
\end{cases}\Longleftrightarrow \begin{cases}
(1-2\varepsilon) (t^2-4s^2x^2)=0,\\ s^2-t^2=0,
\end{cases} \Longleftrightarrow \begin{cases}
(1-2\varepsilon) (1-4x^2)=0,\\ s^2-t^2=0,
\end{cases}
$$
so, we obtain two different possibilities: $$s^2 = t^2,\quad \varepsilon = \frac12,\quad \text{or} \quad s^2 = t^2,\quad x = \frac12.$$

If we consider $A\neq 0$, then $B=0$ and therefore $v + 2\sqrt{-1}\,x\bar v = - t^2-2\sqrt{-1}\,s^2 x$.  Substituting this expression in $D$ we get that $D=2(t^2+2\sqrt{-1}x^2x)\neq0$, which implies that $C=0$.  Observe that system $B=C=0$ is the same as $A=D=0$ just changing the sign of $v$.  One get the same solutions as before.

To finish the proof, we study first the case $s^2 = t^2$ and $\varepsilon = \frac12$.
Directly, $$R^{\frac12}_{121\bar{2}} = \frac{t^2(2x-\sqrt{-1})}{16x}\neq 0.$$  On the other hand, if $s^2 = t^2$ and $x = \frac12$,
$B^{\varepsilon}_{2\bar{2}3\bar{1}} =\frac{\sqrt{-1}\varepsilon(2\varepsilon-1)t^4}{\bar u+\sqrt{-1}\bar z}\neq 0$.
\end{proof}

\subsubsection{Solvmanifolds in Families (Siii1), (Siii3), (Siii4)}
Recall that from Section~\ref{subsec:complex-st}, the Lie algebras underlying (Siii1), (Siii3), and (Siii4) are, respectively, $\frak g_4$, $\frak g_6$, and $\frak g_7$.
In order to give an unified argument, we will gather the complex equations as follows:
$$d\varphi^1=\sqrt{-1}\,(\varphi^{13} + \varphi^{1\bar3}),\quad d\varphi^2 = -\sqrt{-1}\,(\varphi^{23} + \varphi^{2\bar3}),$$
$$ 	d\varphi^3 = x\,\varphi^{1\bar1} +y\,\varphi^{2\bar 2},$$
where $(x,y) = (\pm 1, 0)$ for $\frak g_4$,\, $(x,y) = (1,1)$ for $\frak g_6$ and $(x,y=-x) = (\pm 1, \mp1)$ for $\frak g_7$.
In particular, $x\neq 0$.

\begin{prop}
The Chern connection $\nabla^{Ch}$ is never K\"ahler-like.
\end{prop}

\begin{proof}
Using \cite[Theorem 1.3]{yang-zheng}, a necessary condition for $\nabla^{Ch}$ to be K\"ahler-like is that $\omega$ must be a balanced metric.  A direct computation shows that $\omega$ is balanced if and only if $v=z=0$ and $s^2 x + r^2 y = 0$.  Observe that this last equation has no solutions on $\frak g_4$ and $\frak g_6$, which implies that $\nabla^{Ch}$ is never K\"ahler-like for these algebras.  On $\frak g_7$, since $x=-y$, the only option is $r^2 = s^2$.  Now, if we compute the Bianchi identity-symmetries $B^0$, we get that $B^0_{1\bar{1}2\bar{2}} = \frac{t^2}{2}\neq 0$.
\end{proof}

\begin{prop}\label{prop-g467}
In the notation as above, consider $\varepsilon\neq 0$.  Then, $\nabla^{\varepsilon}$ is K\"ahler-like if and only if $\frak g$ is isomorphic
to $\frak g_4$, $\varepsilon=\frac12$, and $u=v=z=0$ in \eqref{eq:metric}.
In this case, the Strominger-Bismut connection is non-flat.
\end{prop}

\begin{proof}
From now on, consider $\varepsilon\neq 0$.  Let us focus on Chern-symmetries $R^{\varepsilon}_{121\bar{3}}$ and $R^{\varepsilon}_{122\bar{3}}$:
\begin{eqnarray*}
R^{\varepsilon}_{121\bar{3}} &=& \frac{-\varepsilon\, z}{8\sqrt{-1}\, \det\Omega}\, \left(v\,\left[\sqrt{-1}\,(2\varepsilon + 1) uv x + z \left(s^2 x (2\varepsilon + 1) - r^2y (2\varepsilon - 1)\right)\right] \right. \\
&& \left. + \sqrt{-1}\,(2\varepsilon - 1) y \bar u z^2\right), \\
R^{\varepsilon}_{122\bar{3}} &=& \frac{\varepsilon\, v}{8\sqrt{-1}\, \det\Omega}\, \left(v\,\left[\sqrt{-1}\,(2\varepsilon - 1) uv x + z \left(s^2 x (2\varepsilon - 1) - r^2y (2\varepsilon + 1)\right)\right] \right. \\
&& \left. + \sqrt{-1}\,(2\varepsilon + 1) y \bar u z^2\right).
\end{eqnarray*}

Trivially, this two elements vanishes if $v=z=0$. Suppose that $vz\neq 0$.  Then, the system $R^{\varepsilon}_{121\bar{3}} = R^{\varepsilon}_{122\bar{3}}=0$ is equivalent to
$$
\left\{\begin{array}{rcl}
A &:=& v\,\left[\sqrt{-1}\,(2\varepsilon + 1) uv x + z \left(s^2 x (2\varepsilon + 1) - r^2y (2\varepsilon - 1)\right)\right] \\
&& + \sqrt{-1}\,(2\varepsilon - 1) y \bar u z^2 =0,\\[5pt]
B &:=& v\,\left[\sqrt{-1}\,(2\varepsilon - 1) uv x + z \left(s^2 x (2\varepsilon - 1) - r^2y (2\varepsilon + 1)\right)\right] \\
&& + \sqrt{-1}\,(2\varepsilon + 1) y \bar u z^2=0.
\end{array}\right.
$$
Consider the following homogeneous system, whose solutions are precisely $A=B=0$:
$$
\begin{cases}
A(1+2\varepsilon) + B(1-2\varepsilon)=0,\\[5pt]
A(1-2\varepsilon) + B(1+2\varepsilon)=0.
\end{cases}
$$
This last system reduces to
\begin{eqnarray*}
\begin{cases}
8\,\varepsilon\,v\,x\,(\sqrt{-1}uv + s^2z)=0,\\[5pt]
8\,\varepsilon\,\sqrt{-1} \,z\,y \,(\sqrt{-1}r^2v + \bar u z)=0.
\end{cases} &\Longleftrightarrow&
\begin{cases}
\sqrt{-1}uv + s^2z=0,\\[5pt]
y \,(\sqrt{-1}r^2v + \bar u z)=0.
\end{cases} \\
&\Longleftrightarrow&
\begin{cases}
u = \dfrac{\sqrt{-1}\,s^2\,z}{v},\\[5pt]
\dfrac{\sqrt{-1}\,y}{\bar v} \,(r^2|v|^2 - s^2|z|^2)=0.
\end{cases}
\end{eqnarray*}
Moreover, substituting $u = \dfrac{\sqrt{-1}\,s^2\,z}{v}$ in \eqref{det}, we get:
$$
8\sqrt{-1}\, \det\Omega  = \frac{(s^2t^2-|v|^2)(r^2|v|^2 - s^2|z|^2)}{|v|^2}\neq 0.
$$
So, we are forced to consider $y=0$.
However, in this particular case, $R^{\varepsilon}_{123\bar{2}} = \frac{4\sqrt{-1}\,\varepsilon^2 s^4 z}{s^2t^2-|v|^2}\neq 0$.

Let us study now the case when $vz=0$.  Suppose first $z=0$ and $v\neq 0$.  Now, $R^{\varepsilon}_{121\bar{2}}=\frac{-\varepsilon^2u^2v^2x}{2\det\Omega}=0$ if and
only if $u=0$.  But imposing this condition,	 we get that $R^{\varepsilon}_{231\bar{1}}=-\varepsilon v x\neq 0$.

On the other hand, if $v=0$ and $z\neq 0$, symmetry $R^{\varepsilon}_{121\bar{3}}=\frac{-\varepsilon (2 \varepsilon-1 ) y z^3 \bar u}{8\det\Omega}=0$ is equivalent to $(2\varepsilon-1) u y = 0$.
If $\varepsilon = \frac12$ (corresponding to the Strominger-Bismut connection), then
$$
\begin{cases}
R^{\frac12}_{132\bar{2}}=\frac{yz}{2}=0,\\[5pt]
R^{\frac12}_{121\bar{1}}=\frac{s^2z^2 x \bar u}{8\sqrt{-1}\det \Omega}=0,
\end{cases} \Longleftrightarrow \quad u=y=0.
$$
But now, $R^{\frac12}_{133\bar{3}}=\frac{\sqrt{-1}z}{2}\neq 0$.

Suppose now $\varepsilon\neq \frac12$.  If $u=0$,  directly from $R^{\varepsilon}_{132\bar{2}}=\varepsilon yz$ one gets $y=0$.
Analogously, if we suppose $y=0$, then $R^{\varepsilon}_{121\bar{1}}=\frac{-\sqrt{-1} \varepsilon s^2 z^2 \bar u x}{4\det\Omega}$ implies $u=0$.
So, if $\varepsilon\neq \frac12$ we may assume $u=y=0$.
At this point, observe that $R^{\varepsilon}_{131\bar{3}}=\frac{\varepsilon(2\varepsilon-1)(r^2-\sqrt{-1}xt^2)z^2}{8\sqrt{-1}\det \Omega}=0$
if and only if $\varepsilon=\frac12$.

Finally, if $v=z=0$, then
$$
R^{\varepsilon}_{133\bar{1}}=\frac{\varepsilon ((-1 + 2 \varepsilon) s^2 t^4 x^2 + (8 \varepsilon r^2 u + 2 \sqrt{-1} t^2 u x) \bar u)}{8it^2\det\Omega}=0
$$
if and only if the complex number
$$
(2\varepsilon-1) s^2 t^4 x^2 + 8\varepsilon r^2 |u|^2 + \sqrt{-1} (2t^2 x |u|^2)=0.
$$
Observe that the imaginary part vanishes if and only if $u=0$.  Furthermore, one can check that if $u=0$, then all the Chern-symmetries are
zero if and only if $\varepsilon = \frac12$.  As for Bianchi identity-symmetries, in the case $u=v=z=0$ and $\varepsilon = \frac12$, we
have $R^{\frac12}_{ij\bullet\bullet}=R^{\frac12}_{\bullet\bullet k\ell}=0$ for any $i,j,k,\ell\in\{1,2,3\}$ if and only if $y=0$, so the Lie algebra is isomorphic to $\frak g_4$.

Finally, note that $ R^\frac12_{1\bar11\bar1} = t^2 $, whence the last statement.
\end{proof}

\begin{rmk}
We notice that any diagonal metric on $\frak g_4$ gives examples of a non-flat K\"ahler-like Strominger-Bismut connection on a non-K\"ahler manifold. (For the existence of lattices, see \cite[Proposition 2.10]{fino-otal-ugarte}.)
\end{rmk}

\subsubsection{Solvmanifolds in Family (Siii2)}
By Section~\ref{subsec:complex-st}, the underlying Lie algebra is $\frak g_5$.

\begin{prop}
The connection $\nabla^{\varepsilon}$ is never K\"ahler-like.
\end{prop}

\begin{proof}
Let us suppose $\varepsilon=0$. Using \cite[Theorem 1.3]{yang-zheng}, a necessary condition for $\nabla^{Ch}$ to be K\"ahler-like is that $\omega$ must be a balanced metric.  According to \cite[Theorem 4.5]{fino-otal-ugarte}, $\omega$ is balanced if and only if $v=z=0$ and $u\in\mathbb R$.  Now, with these restrictions, $$B^0_{1\bar{1}2\bar{1}} = -r^2\neq 0,$$
whence the statement.

From now on, consider $\varepsilon\neq 0$ and a generic Hermitian metric $\omega$.  Let us focus on Chern-symmetries $R^{\varepsilon}_{121\bar{3}}$ and $R^{\varepsilon}_{122\bar{3}}$ and develop a similar argument as that in Proposition~\ref{prop-g467}:
\begin{eqnarray*}
R^{\varepsilon}_{121\bar{3}} &=& \frac{\varepsilon\, z}{8\sqrt{-1}\, \det\Omega}\, \left(\sqrt{-1}\,\left[(2\varepsilon + 1) r^2 v^2-z (2\varepsilon - 1)(\sqrt{-1}uv + s^2z)\right] \right. \\
&& \left. + (2\varepsilon + 1)  \bar u v z\right),\\
R^{\varepsilon}_{122\bar{3}} &=& \frac{-\varepsilon\, v}{8\sqrt{-1}\, \det\Omega}\, \left(\sqrt{-1}\,\left[(2\varepsilon - 1) r^2 v^2-z (2\varepsilon + 1)(\sqrt{-1}uv + s^2z)\right] \right. \\
&& \left. + (2\varepsilon - 1)  \bar u v z\right).
\end{eqnarray*}

Clearly, this two elements vanishes if $v=z=0$. Suppose that $vz\neq 0$.  Then, the system $R^{\varepsilon}_{121\bar{3}} = R^{\varepsilon}_{122\bar{3}}=0$ is equivalent to
$$
\begin{cases}
A := \sqrt{-1}\,\left[(2\varepsilon + 1) r^2 v^2-z (2\varepsilon - 1)(\sqrt{-1}uv + s^2z)\right] + (2\varepsilon + 1)  \bar u v z =0,\\[5pt]
B := \sqrt{-1}\,\left[(2\varepsilon - 1) r^2 v^2-z (2\varepsilon + 1)(\sqrt{-1}uv + s^2z)\right] + (2\varepsilon - 1)  \bar u v z=0.
\end{cases}
$$
Consider the following homogeneous system, whose solutions are precisely $A=B=0$:
$$
\begin{cases}
A(1+2\varepsilon) + B(1-2\varepsilon)=0,\\[5pt]
A(1-2\varepsilon) + B(1+2\varepsilon)=0.
\end{cases}
$$
This last system is equivalent to
$$
\begin{cases}
\sqrt{-1}uv + s^2z=0,\\[5pt]
\sqrt{-1}r^2v + \bar u z=0.
\end{cases} \Longleftrightarrow \quad
v=z=0,$$
which contradicts our initial hypothesis.

We can suppose now that $vz=0$.  In fact, in this case $v=z=0$ (use $R^{\varepsilon}_{122\bar{2}}=\frac{\sqrt{-1}\varepsilon^2 s^4 z^2}{2\det\Omega}=0$ when $v=0$
and $R^{\varepsilon}_{121\bar{1}}=\frac{\sqrt{-1}\varepsilon^2 r^4v^2}{2\det\Omega}$ if $z=0$).  Now, $B^{\varepsilon}_{1\bar{1}2\bar{1}}=(4\varepsilon-1) r^2=0$ if and only if $\varepsilon=\frac14$ but now $B^{\frac14}_{1\bar{1}3\bar{3}}=-r^2\neq 0$.
\end{proof}

\subsubsection{Solvmanifolds in Families (Siv1), (Siv2), and (Siv3)}
Recall that by Section~\ref{subsec:complex-st}, the underlying Lie algebra to these complex structures is $\frak g_8$.

\begin{prop}
For the holomorphically-parallelizable structure in Family (Siv1), we have that the Chern connection $\nabla^{Ch}$ is flat
and $\nabla^{\varepsilon\neq 0}$ is never K\"ahler-like.
\end{prop}

\begin{proof} The result about Chern connection comes from the fact that the structure is complex-parallelizable.
  When $\varepsilon\neq 0$, one can see that $R^{\varepsilon}_{ij\bullet\bullet}=R^{\varepsilon}_{\bullet\bullet k\ell}\equiv 0$ for any $i,j,k,\ell\in\{1,2,3\}$ if and only if $\varepsilon=\frac12$ (see for example
  $R^\varepsilon_{133\bar1}=-(2\varepsilon-1)\varepsilon r^2$).  But now, $B^{\frac12}_{1\bar{1}2\bar{2}} = \frac{r^4 s^4 - |u|^4}{16\sqrt{-1}\det\Omega}$ if and only if $r^4s^4-|u|^4=0$ which is a contradiction.
\end{proof}

\begin{prop}
For the complex structures in Family (Siv2), $\nabla^{\varepsilon}$ is never K\"ahler-like.
\end{prop}

\begin{proof}
The result for the Chern connection (namely $\varepsilon=0$) follows directly observing that complex structures in Family (Siv2) do not
admit balanced metrics and by \cite[Theorem 1.3]{yang-zheng}.

If $\varepsilon\neq 0$, then
$$
R^{\varepsilon}_{123\bar{1}}=-\frac{\varepsilon^2(r^2s^2-|u|^2)(x r^2s^2 + x|u|^2 + 2\sqrt{-1} r^2 \bar u)}{2 \det \Omega},
$$
and
$$
R^{\varepsilon}_{123\bar{2}}=-\frac{\varepsilon^2(r^2s^2-|u|^2)(r^2s^2 - 2\sqrt{-1} x s^2 u + |u|^2)}{2 \det \Omega}.
$$
Hence, $R^{\varepsilon}_{123\bar{1}}=R^{\varepsilon}_{123\bar{2}}=0$ if and only if
$$
xr^2s^2+x|u|^2 + 2\sqrt{-1} r^2 \bar u=0,
\quad
r^2s^2 + |u|^2 - 2\sqrt{-1} x s^2 u =0.$$
If $x =1$, these equations imply $\Re u=0$, $\Im u=-r^2$, and $r^2=s^2$, which is a contradiction to the positive
definiteness of the metric. Hence, $x=0$ and
$R^{\varepsilon}_{123\bar{2}}$ is always different from zero.
\end{proof}

\begin{prop}
For the complex structures in Family (Siv3), $\nabla^{\varepsilon}$ is K\"ahler-like if and only if $\varepsilon=0$ and $\omega$ is diagonal.
Moreover, in these cases, $\nabla^{Ch}$ is flat.
\end{prop}

\begin{proof}
Let us start with the case $\varepsilon=0$ and imposing $\omega$ to be balanced.  In this situation, $v=z=0$.  Moreover, $B^{0}\equiv 0$ if and only if $u=0$ (see for example $B^0_{1\bar33\bar1}=\frac{2r^2|u|^2}{r^2s^2-|u|^2}$).  Observe that we do not get restriction on the parameter $A$.

Finally, if $\varepsilon\neq 0$, we focus on Bianchi identity-symmetry $B^{\varepsilon}_{1\bar{1}2\bar{2}}$.  Observe that
$$B^{\varepsilon}_{1\bar{1}2\bar{2}} = \frac{2\varepsilon^2(r^2s^2-|u|^2) (r^2s^2|A-1|^2 + |u|^2 |A+1|^2)}{8\sqrt{-1}\, \det\Omega}\neq 0,$$
whence the statement.
\end{proof}

\subsubsection{Solvmanifolds in Family (Sv)}
By Section~\ref{subsec:complex-st}, the underlying Lie algebra is $\frak g_9$.

\begin{prop}
The connection $\nabla^{\varepsilon}$ is never K\"ahler-like.
\end{prop}

\begin{proof}
If $\varepsilon=0$, then the result follows directly by observing that $\frak g_9$ does not admit balanced metrics
and by \cite[Theorem 1.3]{yang-zheng}.

From now on, consider $\varepsilon\neq 0$. Let us focus on the Chern-symmetry $R^{\varepsilon}_{121\bar{2}}$:
$$R^{\varepsilon}_{121\bar{2}} = \frac{\bar v\left[(2r^2s^2-|u|^2)v - (\sqrt{-1}s^2\bar u)z\right]}{32\, \sqrt{-1} \det \Omega}.$$
Hence,
$$R^{\varepsilon}_{121\bar{2}} = 0 \quad \Longleftrightarrow \quad \bar v\left[(2r^2s^2-|u|^2)v - (\sqrt{-1}s^2\bar u)z\right]=0.$$
Suppose first that $v\neq 0$.  Then, we can express $$v=\frac{\sqrt{-1}s^2\bar u z}{2r^2s^2-|u|^2}.$$ Observe that, since $v\neq 0$, also $uz\neq 0$.
Using this particular value for $v$, $R^{\varepsilon}_{122\bar{1}}=\frac{\varepsilon(1-2\varepsilon)s^2\bar u^2}{4(r²s^2-|u|^2)}=0$ if and only if $\varepsilon = \frac12$, but now $R^{\frac12}_{131\bar{2}}=\frac{s^2u\bar z}{4(r^2s^2-|u|^2)}\neq0$.

On the other hand, if $v=0$, then $R^{\varepsilon}_{122\bar{3}}=\frac{\varepsilon(1-2\varepsilon)s^2z\bar u^2}{16\sqrt{-1}\det\Omega}=0$ holds if and only if $(2\varepsilon-1)uz=0$.
None of the three cases leads to K\"ahler-like connections: if $\varepsilon = \frac12$, then $R^{\frac12}_{233\bar{3}}=\frac{s^4|z|^2}{32\sqrt{-1}\det\Omega}=0$ to
conclude that $z=0$ but then $R^{\frac12}_{123\bar{1}}=\frac{-(r^2s^2-|u|^2)}{4t^2}\neq 0$; if $z=0$,
then $R^{\varepsilon}_{131\bar{3}}=\frac{\varepsilon(2\varepsilon-1)r^4}{t^2}$ vanishes if and only if $\varepsilon = \frac12$,
that has already been studied; finally, if $u=0$, then take $R^{\varepsilon}_{121\bar{1}}=\frac{-\varepsilon r^2s^2z}{4(r^2s^2-|z|^2)}$.
This element is zero if and only if $z=0$, which has already been analysed.
\end{proof}

\section{K\"ahler-like Levi-Civita connection}

\subsection{Hermitian condition along the Ricci-flow}
In this section, we prove that the K\"ahler-like condition for the Levi-Civita connection assures that the Ricci flow preserves the Hermitian condition of the initial metric along analytic solutions.
We wonder whether, as an application, one may prove that non-toral nilmanifolds do not admit
invariant Hermitian metrics whose Levi-Civita connection is K\"ahler-like, thanks to the results by Lauret \cite{lauret} on the Ricci flow for nilpotent Lie groups.
In fact, we will prove that this is true in dimension six, see Proposition \ref{prop:lc}.

\begin{thm}\label{thm:ricci-flow-hermitian}
Let $g_0$ be a Hermitian metric on a compact complex manifold, and consider an analytic solution $(g(t))_{t\in(-\varepsilon,\varepsilon)}$ for $\varepsilon>0$ of the Ricci flow
\begin{equation}\label{eq:ricci-flow}\tag{RF}
\frac{d}{dt}g(t) = -\mathrm{Ric}(g(t)), \qquad g(0)=g_0 .
\end{equation}
If the Levi-Civita connection of $g_0$ is K\"ahler-like, then $g(t)$ is Hermitian for any $t$.
\end{thm}

\begin{proof}
The proof follows  by the evolution equation for the curvatures along the Ricci flow in \cite[Corollary 7.3 and Theorem 7.1]{hamilton}:
\begin{eqnarray*}
\frac{d}{dt} \mathrm{Ric}(t)_{JK} &=& \Delta_{t} \mathrm{Ric}(t)_{JK}+2g(t)^{PQ}g(t)^{RS}R(t)_{PJKR}\mathrm{Ric}(t)_{QS} \\
&& -2g(t)^{PQ}\mathrm{Ric}(t)_{JP}\mathrm{Ric}(t)_{QK} , \\
\frac{d}{dt} R(t)_{IJKL} &=& \Delta_{t} R(t)_{IJKL} \\
&& + 2 \left( g(t)^{PR} g(t)^{QS} R(t)_{IPJQ} R(t)_{KRLS} \right. \\
&& \left. - g(t)^{PR} g(t)^{QS} R(t)_{IPJQ} R(t)_{LRKS} \right. \\
&& \left. - g(t)^{PR} g(t)^{QS} R(t)_{IPLQ} R(t)_{JRKS} \right. \\
&& \left. + g(t)^{PR} g(t)^{QS} R(t)_{IPKQ} R(t)_{JRLS} \right) \\
&& -g(t)^{PQ} \left( R(t)_{PJKL} \mathrm{Ric}(t)_{QI} + R(t)_{IPKL} \mathrm{Ric}(t)_{QJ} \right. \\
&& \left. + R(t)_{IJPL} \mathrm{Ric}(t)_{QK} + R(t)_{IJKP} \mathrm{Ric}(t)_{QL} \right) .
\end{eqnarray*}
Here, capital letters $I$, $J$, \dots vary in $\{1, \ldots, n, \bar1, \ldots, \bar n\}$ while small letters $i$, $j$, \dots vary in $\{1,\ldots,n\}$, and for simplicity we just write the dependence in $t$ meaning to refer to the quantities for $g(t)$. The dot denotes the derivative at zero, {\itshape e.g.} $\dot{\mathrm{Ric}}_{jk} = \left.\frac{d}{dt}\right\lfloor_{t=0} \mathrm{Ric}_{jk}(t)$.

The K\"ahler-like condition assures that
$$ R_{ijKL}(0)=0, \qquad \mathrm{Ric}_{ij}(0)=0 . $$
This, together with the formula for the evolution of the Ricci curvature tensor, whence
$$ \dot{\mathrm{Ric}}_{jk} = 0 , $$
yields that
$$ \dot g_{jk} = \ddot g_{jk} = 0 . $$

Moreover, notice that the K\"ahler-like assumption forces $R_{IJKL}(0)$ to be zero whenever three or more of the indices in $(I,J,K,L)$ have the same type. The same holds for the tensor $g(0)^{PR}g(0)^{QS}R(0)_{PIQJ}R(0)_{RKSL}$, that appears in the right-hand-side for $\dot R$. From the evolution equation of the Riemann curvature tensor, we get the same property also for $\dot{R}_{IJKL}$. Then we get
$$ \dddot{g}_{jk}=0 $$
by straightforward computations.

Further derivatives behave similarly. Indeed, the only remark to be noticed is that $\dot{R}_{IJKL}$ involves the tensors $R_{\bullet\bullet\bullet\bullet}(0)$, $\dot{R}_{\bullet\bullet\bullet\bullet}$, $\mathrm{Ric}_{\bullet\bullet}(0)$, $\dot{\mathrm{Ric}}_{\bullet\bullet}$, other than the Hermitian metric $g_{\bullet\bullet}(0)$ and its time-derivative $\dot{g}_{\bullet\bullet}$, whence it still satisfies the property to vanish whenever three out of $(I,J,K,L)$ have the same type.

So we get that the solution $g(t)_{jk}$ has vanishing time-derivatives of any order at $0$. Being analytic, it is constant, so $g(t)_{jk}=g(0)_{jk}=0$, that is, $g(t)$ is still Hermitian.
\end{proof}

\subsection{K\"ahler-like Levi-Civita connection on six-dimensional Calabi-Yau solvmanifolds}\label{LCKL}
In this section, we prove the following:
\begin{prop}\label{prop:lc}
Let $(X,J,g)$ be a complex six-dimensional solvmanifold with holomorphically-trivial canonical bundle endowed with an
invariant Hermitian metric. The Levi-Civita connection is K\"ahler-like if and only if the metric is K\"ahler.
\end{prop}

\begin{proof}
First, consider the nilpotent case.
We provide explicit computations.
Without loss of generality, we can reduce to the case when the Hermitian metric is balanced, thanks to \cite[Theorem 1.3]{yang-zheng}, and we can exclude the holomorphically-parallelizable case (Np) thanks to the fact that holomorphically-parallelizable manifolds are Chern-flat and to \cite[Theorem 1.2]{yang-zheng}. So we only remain with Family (Ni) with the exception of Lie algebra $\mathfrak h_8$, and the Family (Niii) in fact only for $\mathfrak h_{19}^-$.
We distinguish complex structures of nilpotent, respectively non-nilpotent type, see \cite{ugarte-villacampa}.

For nilpotent complex structures on nilpotent Lie algebras, argue as follows. First of all, the generic Hermitian metric as in \eqref{eq:metric} is balanced if and only if it is equivalent to one with parameters $r^2=1$, $v=z=0$, and also $s^2+D=\sqrt{-1}\bar{u}\lambda$, see \cite{ugarte-villacampa}.
We compute, for example, $R^{LC}_{1\bar112}=-\frac{3}{8}\,\rho\, t^2$.
So the Levi-Civita connection of a balanced metric on a nilmanifold with complex structure in Family (Ni) with $\rho=1$ is never K\"ahler-like.
This forces $\rho=0$, and the Lie algebra to be either $\mathfrak{h}_3$ (with $\lambda=0$ and $D=-1$) or $\mathfrak{h}_5$ (with $\lambda=1$ and $D\in[0,\frac{1}{4})$).
Moreover, since $R^{LC}_{3\bar213}=\frac{1}{8}\,\frac{\sqrt{-1}\, t^4\, u\, \bar{D}}{s^2 - |u|^2}$, then either $u=0$ (case $\mathfrak{h}_3$) or $D=0$ (case $\mathfrak{h}_5$).

In the first case, that is, when $\rho=\lambda=0$ and $D=-1$, we compute $B^{LC}_{1\bar13\bar3} = \frac{{\left(\lambda^{2} + \rho^{2} + s^{2} + \sqrt{-1} \, \lambda u - \sqrt{-1} \, \lambda \overline{u}\right)} t^{4}}{8 \, {\left(s^{2} - u \overline{u}\right)}}$ that reduces to $B^{LC}_{1\bar13\bar3} = \frac{{s^{2}} t^{4}}{8 \, {\left(s^{2} - u \overline{u}\right)}}\neq0$, in fact, $B^{LC}_{1\bar13\bar3} = \frac{t^{4}}{8}\neq0$, by using $s^2=1$ and $u=0$. Therefore, in this case, the Levi-Civita connection is not K\"ahler-like.

In the second case, that is, when $\rho=D=0$ and $\lambda=1$, we compute
$B^{LC}_{1\bar12\bar2} = -\frac{1}{8} \, {\left(\lambda^{2} + 2 \, \rho^{2} - D - \overline{D}\right)} t^{2}$ that reduces to $B^{LC}_{1\bar12\bar2} = -\frac{1}{8} \, t^{2}\neq 0$. Therefore, also in this second case,
the Levi-Civita connection is not K\"ahler-like.

Finally, for non-nilpotent complex structure on nilpotent Lie algebras, argue as follows.
The generic Hermitian balanced metric as in \eqref{eq:metric} forces $\nu=0$ (case $\mathfrak h_{19}^-$) and can be reduced to parameters $u=z=0$, (and moreover either $v=1$, or $v=0$ and $t^2=1$,) see \cite{ugarte-villacampa}.
We compute, for example, $R^{LC}_{3\bar22\bar3}=-\frac{s^4+t^4}{8r^2}\neq 0$ and $R^{LC}_{2\bar23\bar3}=0$. Therefore the Levi-Civita connection is not K\"ahler-like.

We restrict now our attention to the solvable non-nilpotent case.

We remind that by~\cite[Theorem 1.3]{yang-zheng}, the Levi-Civita connection be K\"ahler-like implies that the metric is balanced.
Hence and by~\cite{fino-otal-ugarte} we restrict our attention to the Lie algebras $\frak g_1$, $\frak g_2^{\alpha}$, $\frak g_3$, $\frak g_5$, $\frak g_7$, and $\frak g_8$.

If the Lie algebra is isomorphic to $\frak g_1$ or $\frak g_2^{\alpha}$ (case (Si)) then:
\begin{eqnarray*}
\lefteqn{ R^{\text{LC}}_{12\bar{1}\bar{2}}=\frac{(A^2+\bar A^2)(r^2s^2-|u|^2)+2|A|^2(r^2s^2+|u|^2)}{4t^2}=0 } \\
&\Leftrightarrow& (\Re A)^2r^2s^2+(\Im A)^2|u|^2=0,\\
&\Leftrightarrow& \Re A=0,\, u=0, \\
&\Leftrightarrow&A=\sqrt{-1},\, u=0.
\end{eqnarray*}
That is, the Lie algebra is $\frak g_2^0$ and $\omega$ is K\"ahler.

If the Lie algebra is isomorphic to $\frak g_3$ (case (Sii)) then:
$$
R^{\text{LC}}_{23\bar{2}\bar{3}}=\frac{(1+4x^2)(t^4+4s^4x^2)}{128r^2x^2}\neq0.
$$

If the Lie algebra is isomorphic to $\frak g_5$ (case (Siii2)) then:
$$
\begin{array}{lclcl}
R^{\text{LC}}_{13\bar{2}\bar{3}}=\frac{-\sqrt{-1}(4r^2s^2+t^4)u}{8(r^2s^2-u^2)}=0 &\Leftrightarrow& \sqrt{-1}(4r^2s^2+t^4)u=0&\Leftrightarrow&u=0.
\end{array}
$$
But $u=0$ implies $R^{\text{LC}}_{233\bar{1}}=\frac{-t^2}{2}\neq0$.

If the Lie algebra is isomorphic to $\frak g_7$ (case (Siii4)) then:
$$
R^{\text{LC}}_{23\bar{2}\bar{3}}=\frac{s^2(t^4+4|u|^2)}{8(s^4-|u|^2)}\neq0.
$$
If the Lie algebra is isomorphic to $\frak g_8$, then the complex structure belongs to Families (Siv1) or (Siv3).
In the first case we directly find that $R^{\text{LC}}_{133\bar{1}}=\frac{r^2}{8}\neq 0$.
In the second case the K\"ahler-like condition implies the following system:
$$
\begin{array}{lcl}
R^{\text{LC}}_{133\bar{1}}&=&\frac{r^2\left[(A-1)^2r^2s^2-(A^2-6A+1)|u|^2\right]}{8(r^2s^2-|u|^2)}=0,\\
R^{\text{LC}}_{133\bar{2}}&=&\frac{-\sqrt{-1}u\left[(A^2+6A+1)r^2s^2-(A+1)^2|u|^2\right]}{8(r^2s^2-|u|^2)}=0.
\end{array}
$$
If $u=0$ then the first equation implies $(A-1)^2r^2s^2=0$ which is not possible. On the other hand if
$u\neq 0$ then the vanishing of the brackets corresponds to a homogeneous linear system in $r^2s^2$ and
$|u|^2$. The equations are linear dependent if and only if $A=0$, but in this case the second equation reduces to $r^2s^2-|u|^2=0$.
\end{proof}

\appendix

\renewcommand\thesection{\Roman{section}}

\section{Bianchi and Chern symmetries for the Gauduchon connections on nilmanifolds in Family (Ni)}\label{app:BC-Ni}

\begin{eqnarray*}
R^{\varepsilon}_{121\bar{1}}& =&2\,t^2\,\varepsilon (1-\varepsilon) \rho,\\[5pt]
R^{\varepsilon}_{122\bar{1}}& =&\lambda\,R^{\varepsilon}_{121\bar{1}},\\[5pt]
R^{\varepsilon}_{122\bar{2}}& =&\bar D\,R^{\varepsilon}_{121\bar{1}},\\[5pt]
R^{\varepsilon}_{123\bar{3}}& =&\frac{t^4\,\varepsilon}{s^2-|u|^2}\rho(2\varepsilon-1) (s^2+\sqrt{-1}\lambda u + \bar D),\\[5pt]
R^{\varepsilon}_{131\bar{3}} & =& \frac{-2\sqrt{-1}\,t^4\,\varepsilon^2}{s^2-|u|^2} \rho\, u,\\[5pt]
R^{\varepsilon}_{132\bar{3}} & =& \frac{-2\,t^4\,\varepsilon^2}{s^2-|u|^2}  \rho\, (\sqrt{-1}\lambda u + \bar D),\\[5pt]
R^{\varepsilon}_{133\bar{1}} & =& \frac{-t^4\,\varepsilon}{s^2-|u|^2} (2\varepsilon-1) (s^2+\sqrt{-1}\lambda u),\\[5pt]
R^{\varepsilon}_{133\bar{2}} & =& \frac{-\sqrt{-1}\,t^4\,\varepsilon}{s^2-|u|^2} (2\varepsilon-1)\,u\, \bar D,\\[5pt]
R^{\varepsilon}_{231\bar{3}} & =& \frac{2\,s^2\,t^4\,\varepsilon^2}{s^2-|u|^2}  \rho ,\\[5pt]
R^{\varepsilon}_{232\bar{3}} & =& \frac{2\,t^4\,\varepsilon^2}{s^2-|u|^2} \rho\, (\lambda s^2-\sqrt{-1}\bar u \bar D),\\[5pt]
R^{\varepsilon}_{233\bar{1}} & =& \frac{-t^4\,\varepsilon}{s^2-|u|^2} (2\varepsilon-1) (\lambda (s^2+\sqrt{-1}\lambda u + \bar D)-\sqrt{-1}\bar u \bar D),\\[5pt]
R^{\varepsilon}_{233\bar{2}} & =& \frac{-t^4\,\varepsilon}{s^2-|u|^2} (2\varepsilon-1)\bar D (\sqrt{-1}\lambda u + \bar D).\\[5pt]
B^{\varepsilon}_{1\bar{1}2\bar{1}}& =& \frac{-t^2}{2} (2\varepsilon-1)^2\,\lambda,\\[5pt]
B^{\varepsilon}_{1\bar{1}2\bar{2}}& =& \frac{-t^2}{2} \left[2\varepsilon (\lambda^2 - D + 4\rho \varepsilon) + \bar D (4\varepsilon^2 - 6 \varepsilon + 1)\right],\\[5pt]
B^{\varepsilon}_{1\bar{1}3\bar{3}}& =&\frac{t^4}{2 (s^2-|u|^2)} \left[4\,\rho\varepsilon^2 -s^2 (2\varepsilon-1)(4\varepsilon - 1)\right] ,\\[5pt]
B^{\varepsilon}_{1\bar{2}2\bar{1}}& =& \frac{t^2}{2} \left[2\varepsilon (4\rho \varepsilon-\bar D) + (D-\lambda^2) (4\varepsilon^2 - 6 \varepsilon + 1)\right],\\[5pt]
B^{\varepsilon}_{1\bar{2}2\bar{2}}& =& \bar D B^{\varepsilon}_{1\bar{1}2\bar{1}},\\[5pt]
B^{\varepsilon}_{1\bar{2}3\bar{3}}& =& \frac{-t^4}{2 (s^2-|u|^2)} \left[4\,\sqrt{-1}\,\rho u \varepsilon^2 +(2\varepsilon-1)(4\varepsilon - 1) (\lambda s^2 + \sqrt{-1} u D)\right] ,\\[5pt]
B^{\varepsilon}_{1\bar{3}2\bar{3}}& =& \frac{2\,t^4\,\varepsilon^2}{s^2-|u|^2} \rho  (D+ s^2 - \sqrt{-1} \lambda \bar u) ,\\[5pt]
B^{\varepsilon}_{1\bar{3}3\bar{1}}& =&\frac{-t^4\,\varepsilon}{s^2-|u|^2} \left[s^2 + 2\varepsilon (\lambda^2- s^2 - 2 \lambda \Im u)\right],\\[5pt]
B^{\varepsilon}_{1\bar{3}3\bar{2}}& =&\frac{t^4\,\varepsilon}{s^2-|u|^2} \left[(4\varepsilon-1) (\lambda s^2 +\sqrt{-1} u D) - 2\varepsilon \bar D (\lambda + \sqrt{-1} u)\right],\\[5pt]
B^{\varepsilon}_{2\bar{1}3\bar{3}}& =&\overline{B^{\varepsilon}_{1\bar{2}3\bar{3}}},\\[5pt]
B^{\varepsilon}_{2\bar{2}3\bar{3}}& =&\frac{-t^4}{2(s^2-|u|^2)} \left[(2\varepsilon-1)(4\varepsilon-1) (\lambda^2 s^2 -2\lambda \Im (uD) + |D|^2) - 4s^2\rho\varepsilon^2\right],\\[5pt]
B^{\varepsilon}_{2\bar{3}3\bar{1}}& =&\overline{B^{\varepsilon}_{1\bar{3}3\bar{2}}},\\[5pt]
B^{\varepsilon}_{2\bar{3}3\bar{2}}& =&\frac{t^4\,\varepsilon}{s^2-|u|^2} \left[(2\varepsilon-1)|D|^2 + (4\varepsilon-1)\lambda  (\lambda s^2 -2 \Im (uD))\right].
\end{eqnarray*}

\section{Bianchi symmetries for the Levi-Civita connection on solvmanifolds $\frak g_1$ and $\frak g_2^{\alpha}$ in Family (Si)}\label{sec:app:g2a-Chern}

\begin{eqnarray*}
B^{0}_{1\bar{3}3\bar{1}}& =&\frac{2\,r^2 |u|^2}{r^2s^2-|u|^2},\qquad
B^{0}_{1\bar{3}3\bar{2}}=\frac{-2\,\sqrt{-1}\,r^2 s^2 u}{r^2s^2-|u|^2},\\[5pt]
B^{0}_{2\bar{3}3\bar{1}}& =&\overline{B^{\varepsilon}_{1\bar{3}3\bar{2}}},\qquad \qquad
B^{0}_{2\bar{3}3\bar{2}}=\frac{2\,s^2 |u|^2}{r^2s^2-|u|^2}.
\end{eqnarray*}

\section{Bianchi and Chern symmetries for the Gauduchon connections on solvmanifolds $\frak g_2^0$ in Family (Si)}\label{sec:app:g2a}

\begin{eqnarray*}
R^{\varepsilon}_{131\bar{3}}& =&\frac{\varepsilon (2\varepsilon-1) r^2 s^2 z^2}{8\sqrt{-1}\, \det\Omega},\\[5pt]
R^{\varepsilon}_{132\bar{3}}& =&\frac{-\varepsilon (2\varepsilon-1) r^2 s^2  v z}{8\sqrt{-1}\, \det\Omega},\\[5pt]
R^{\varepsilon}_{133\bar{3}}& =&\frac{\varepsilon\,z}{8 \det\Omega} \left[r^2s^2t^2 - 2\varepsilon r^2|v|^2 + 2(\varepsilon-1)s^2|z|^2\right],\\[5pt]
R^{\varepsilon}_{231\bar{3}}& =&-R^{\varepsilon}_{132\bar{3}},\\[5pt]
R^{\varepsilon}_{232\bar{3}}& =&\frac{\varepsilon (2\varepsilon-1) r^2 s^2 v^2}{8\sqrt{-1}\, \det\Omega},\\[5pt]
R^{\varepsilon}_{233\bar{3}}& =&\frac{\varepsilon\,v}{8 \det\Omega} \left[r^2s^2t^2 + 2(\varepsilon-1) r^2|v|^2 - 2 \varepsilon s^2|z|^2\right].\\[5pt]
B^{\varepsilon}_{1\bar{1}3\bar{3}}& =&\frac{-\varepsilon (2\varepsilon-1) r^2 s^2 |z|^2}{8\sqrt{-1}\, \det\Omega},\\[5pt]
B^{\varepsilon}_{1\bar{2}3\bar{3}}& =&\frac{\varepsilon (2\varepsilon-1) r^2 s^2 \bar v z}{8\sqrt{-1}\, \det\Omega},\\[5pt]
B^{\varepsilon}_{1\bar{3}3\bar{1}}& =&\frac{(2\varepsilon-1)(4\varepsilon-1) r^2 s^2 |z|^2}{16 \sqrt{-1}\, \det\Omega},\\[5pt]
B^{\varepsilon}_{1\bar{3}3\bar{2}}& =&\frac{-(2\varepsilon-1)(4\varepsilon-1) r^2 s^2 \bar v z}{16 \sqrt{-1}\, \det\Omega},\\[5pt]
B^{\varepsilon}_{1\bar{3}3\bar{3}}& =&\frac{z}{16 \det\Omega}\left[(4\varepsilon-1)r^2s^2t^2 + 2 (2\varepsilon^2-4\varepsilon+1) r^2 |v|^2 - 4\varepsilon^2 s^2 |z|^2\right],\\[5pt]
B^{\varepsilon}_{2\bar{1}3\bar{3}}& =&\overline{B^{\varepsilon}_{1\bar{2}3\bar{3}}},\\[5pt]
B^{\varepsilon}_{2\bar{2}3\bar{3}}& =&\frac{-\varepsilon (2\varepsilon-1) r^2 s^2 |v|^2}{8\sqrt{-1}\, \det\Omega},\\[5pt]
B^{\varepsilon}_{2\bar{3}3\bar{1}}& =&\overline{B^{\varepsilon}_{1\bar{3}3\bar{2}}},\\[5pt]
B^{\varepsilon}_{2\bar{3}3\bar{2}}& =&\frac{(2\varepsilon-1)(4\varepsilon-1) r^2 s^2 |v|^2}{16 \sqrt{-1}\, \det\Omega},\\[5pt]
B^{\varepsilon}_{2\bar{3}3\bar{3}}& =&\frac{v}{16 \det\Omega}\left[(4\varepsilon-1)r^2s^2t^2 - 4\varepsilon^2 r^2 |v|^2 + 2 (2\varepsilon^2-4\varepsilon+1)  s^2 |z|^2\right].
\end{eqnarray*}

\end{document}